\documentclass[11pt]{article}
\usepackage{amsthm,amsmath,amssymb,bbm}
\usepackage{color}
\def\bbe{\mathbb E}

\def\bbone{{\mathbbm 1}}

\theoremstyle{plain}
\newtheorem{thm}{Theorem}[section]

\newtheorem{cor}{Corollary}[section]
\newtheorem{prop}{Proposition}[section]
\newtheorem{lem}{Lemma}[section]
\newtheorem{defi}{Definition}[section]
\newtheorem{rem}{Remark}[section]
\numberwithin{equation}{section}

\begin{document}
\title{On Stein's Method for Multivariate Self-Decomposable Laws With Finite First Moment}

\author{Benjamin Arras\thanks{Universit\'e de Lille, Laboratoire Paul Painlev\'e, Batiment M3, Cit\'e Scientifique, 59655 Villeneuve-d'Ascq; arrasbenjamin@gmail.com.}\; and Christian Houdr\'e\thanks{Georgia Institute of Technology, 
School of Mathematics, Atlanta, GA 30332-0160; houdre@math.gatech.edu. Research supported in part by the grants \# 246283 and \# 524678 from the Simons Foundation.  
}}

\maketitle

\vspace{\fill}
\begin{abstract}
We develop a multidimensional Stein methodology for non-degenerate self-decomposable random vectors in $\mathbb{R}^d$ having finite first moment. Building on previous univariate findings, we solve an integro-partial differential Stein equation by a mixture of semigroup and Fourier analytic methods. Then, under a second moment assumption, we introduce a notion of Stein kernel and an associated Stein discrepancy specifically designed for infinitely divisible distributions. Combining these new tools, we obtain quantitative bounds on smooth-Wasserstein distances between a probability measure in $\mathbb{R}^d$ and a non-degenerate self-decomposable target law with finite second moment. Finally, under an appropriate spectral gap assumption, we investigate, via variational methods, the existence of Stein kernels. In particular, this leads to quantitative versions of classical results on characterizations of probability distributions by variational functionals.
\end{abstract}
\vspace{\fill}

\section{Introduction}

Stein's method is a powerful device to quantify proximity in law between random variables. It has proven to be particularly useful to compute explicit rates of convergence for several limiting theorems appearing in probability theory (from the standard central limit theorem to more complex probabilistic models satisfying some specific asymptotic behavior). Moreover, it has been successfully implemented for a large collection of one dimensional target limiting laws (see \cite{Stein1,Stein2,CGS,R} for standard references on the subject and \cite{LRS17} for a more recent survey). All the previously mentioned works essentially focus on the unidimensional setting and related multidimensional results are relatively sparse in the literature. Indeed, the multidimensional Stein's method has mainly been developed for the multivariate normal laws (see e.g. \cite{Bar90,Go1991,GR96,RR96,Rai04,ChM08,RR09,Mec09,NPR10,R13,NPS14}) and for invariant measures of multidimensional diffusions (\cite{MacGor16,GDVM16}). In particular, the work \cite{GDVM16} proposes a general Stein's method framework for target probability measures $\mu$ on $\mathbb{R}^d$, $d\geq1$, which satisfy the following set of assumptions: $\mu$ has finite mean, is absolutely continuous with respect to the $d$-dimensional Lebesgue measure and its density is continuously differentiable with support the whole of $\mathbb{R}^d$. 

Below, we introduce and develop a multidimensional Stein's methodology for a specific class of probability measures on $\mathbb{R}^d$, namely non-degenerate self-decomposable laws with finite first moment (see \eqref{eq:2.1} in Section \ref{sec:not} for a definition). This class of probability measures, introduced by Paul L\'evy in \cite{L54}, is rather natural in the context of limit theorems for sum of independent summands and has been thoroughly studied in several classical books (see e.g. \cite{K38,L54,GK68,Loe78,Pet95,S}). Nevertheless, while being very classical in the context of limit theorems, no systematic Stein's method has been implemented for multivariate non-degenerate self-decomposable distributions. (The whole class of non-degenerate self-decomposable laws with finite first moment is different, but intersects with the class of target probability measures considered in \cite{GDVM16} and covered by their methodology. Indeed, non-degenerate self-decomposable laws with finite first moment admit a Lebesgue density, which might not be differentiable on $\mathbb{R}^d$, and whose support might be a half-space of $\mathbb{R}^d$.) Finally, many classical probability measures on $\mathbb{R}^d$ are self-decomposable (see \cite{S,SV03} and Section \ref{sec:SteinEqu} below for some examples).

From our previous univariate work \cite{AH18}, the multidimensional Stein's method we implement is a generalization of the semigroup method "\`a la Barbour" (\cite{Bar90}). Thanks to the particular structure of self-decomposable characteristic functions, this semigroup approach relies heavily on Fourier analysitic tools. Moreover, the generator of the aforementioned semigroup is an integro-differential operator reflecting the infinite divisibility of the target law and it can be seen as a direct consequence of a characterization identity originating in \cite{HPAS} and further developed and analyzed in \cite{AH18}. The resulting Stein equation is a non-local partial differential equation and contrasts with the usual second order partial differential equations associated with the multivariate Gaussian distribution or with the invariant measures of It\^o diffusions.

Then, we apply our Stein methodology to quantify proximity, in smooth Wasserstein distances of orders $1$ and $2$, between an appropriate probability measure on $\mathbb{R}^d$ and a non-degenerate self-decomposable laws with finite second moment. Key quantities used in our analysis are relevant versions of Stein kernels and of Stein discrepancies in this infinitely divisible setting (see Definition \ref{defSteinkernel2}). Stein kernel and Stein discrepancy are concepts which have mostly been well developed in the Gaussian setting and have recently gained a certain momentum in connection with random matrices (\cite{Cha09}), Malliavin calculus (\cite{NP1,NP3}), functional inequalities (\cite{LNP15,CFP18}), optimal transport (\cite{F18}) and rates of convergence for multidimensional central limit theorems (\cite{NPS14}). In particular, the work \cite{CFP18} investigates the question of existence of a Gaussian Stein kernel for probability measures satisfying a Poincar\'e inequality or a converse weighted Poincar\'e inequality (see e.g. \cite{BGL14} for a definition). Thanks to earlier work on characterizing functionals of infinitely divisible distributions \cite{ChLo87}, we introduce in the last section of the present manuscript the relevant variational setting which ensures the existence of Stein kernel and implies manageable upper bounds on the Stein discrepancy. In particular, Theorem \ref{th:PoinQuan} is a quantitative version of the characterizing results contained in \cite{ChLo87}.

Let us further describe the content of these notes. In  the next section, we introduce the notations used throughout this work. In Section \ref{sec:SteinEqu}, we develop the multidimensional Stein methodology for non-degenerate self-decomposable random vector with finite first moment, extending our univariate approach (\cite{AH18}). In Section \ref{sec:SK}, we introduce the infinitely divisible version of Stein kernel (and of the Stein discrepancy) and study the existence of the latter under an appropriate version of Poincar\'e inequality. We end this section by providing quantitative upper bounds on the smooth Wasserstein distance of order two in terms of Poincar\'e constants and of the second moment of the L\'evy measure of the target self-decomposable distribution. A technical appendix finishes these notes.

\section{Notations}\label{sec:not}
Throughout, let $\|\cdot\|$ and $\langle \cdot;\cdot \rangle$ be respectively the Euclidean norm and the inner product in $\mathbb{R}^d$, $d\geq 1$. Let also ${\cal S}(\mathbb{R}^d)$ be the Schwartz space of infinitely differentiable rapidly decreasing real-valued functions defined on $\mathbb R^d$, and by $\mathcal{F}$ the Fourier transform operator given, for $f\in {\cal S}(\mathbb{R}^d)$, by 
\begin{align*}
\mathcal{F}(f)(\xi)=\int_{\mathbb{R}^d}f(x)e^{-i \langle \xi; x \rangle}dx, \quad \xi \in \mathbb{R}^d.
\end{align*}
On ${\cal S}(\mathbb{R}^d)$, the Fourier transform is an isomorphism and the following inversion formula holds
\begin{align*}
f(x)=\int_{\mathbb{R}^d}\mathcal{F}(f)(\xi)e^{+i \langle \xi; x \rangle}\frac{d\xi}{(2\pi)^d}, \quad x\in \mathbb{R}^d.
\end{align*}
Let ${\cal C}_b(\mathbb{R}^d)$ be the space of bounded continuous functions on $\mathbb{R}^d$ endowed with the uniform norm $\|f\|_\infty=\sup_{x\in \mathbb{R}^d}|f(x)|$, for $f\in {\cal C}_b(\mathbb{R}^d)$. For any bounded linear operator, $T$, from a Banach space $({\cal X}, \|\cdot\|_{{\cal X}})$ to another Banach space $({\cal Y}, \|\cdot\|_{{\cal Y}})$ the operator norm is, as usual,
\begin{align}
\|T\|_{op}=\underset{f\in {\cal X},\, \|f\|_{{\cal X}}\ne 0}{\sup}\dfrac{\|T(f)\|_{{\cal Y}}}{\|f\|_{{\cal X}}}.
\end{align}
More generally, for any $r$-multilinear form $F$ from $(\mathbb{R}^d)^r$, $r\geq 1$, to $\mathbb{R}$, the operator norm of $F$ is
\begin{align}
\|F\|_{op}:=\sup \left(|F(v_1,...,v_r)|:\, v_j \in \mathbb{R}^d,\, \|v_j\|=1,\, j=1,...,r\right).
\end{align}
Throughout, a L\'evy measure is a positive Borel measure on $\mathbb{R}^d$ such that $\nu(\{0\})=0$ and $\int_{\mathbb{R}^d} (1\wedge \|u\|^2)\nu(du)<+\infty$. An $\mathbb{R}^d$-valued random vector $X$ is infinitely divisible with triplet $(b,A,\nu)$ (written $X\sim ID(b, A,\nu)$), if its characteristic function $\varphi$ writes, for all $\xi\in\mathbb{R}^d$, as
\begin{align}
\varphi(\xi)=\exp\left(i \langle b;\xi \rangle-\frac{1}{2}\langle \xi;A(\xi) \rangle +\int_{\mathbb{R}^d}\left(e^{i \langle \xi; u\rangle}-1-i\langle \xi;u \rangle\bbone_{D}(u)\right)\nu(du)
\right),
\end{align}
with $b\in\mathbb{R}^d$, $A$ a symmetric nonnegative definite $d\times d$ matrix, $\nu$ a L\'evy measure on $\mathbb{R}^d$ and $D$ the closed Euclidean unit ball of $\mathbb{R}^d$. In the sequel, we are mainly interested in a subclass of infinitely divisible distributions, namely the self-decomposable laws (SD). If $\varphi$ is the characteristic function of a self-decomposable distribution, then for all $\gamma\in (0,1)$, there exists, on $\mathbb{R}^d$, a probability measure, say $\mu_\gamma$, such that, for all $\xi\in \mathbb{R}^d$
\begin{align}\label{eq:2.1}
\int_{\mathbb{R}^d}e^{i\langle \xi; u \rangle}\mu_\gamma(du)=\dfrac{\varphi(\xi)}{\varphi(\gamma\xi)}
\end{align}
(Recall that by \cite[Lemma $7.5$]{S} $\varphi(\xi)\ne 0$, for all $\xi\in \mathbb{R}^d$). Moreover, by \cite[Theorem 15.10]{S}, the L\'evy measure of a self-decomposable distribution is such that, for any Borel set $B$ in $\mathbb{R}^d\setminus \{0\}$
\begin{align}\label{rep:sd}
\nu(B)=\int_{S^{d-1}}\lambda(dx)\int_0^{+\infty} \bbone_{B}(rx)k_x(r)\frac{dr}{r},
\end{align}
with $\lambda$ a finite positive measure on the Euclidean unit sphere $S^{d-1}$ and $k_x(r)$ a nonnegative function (Lebesgue) measurable in $x\in S^{d-1}$ and decreasing in $r>0$ (namely, $k_x(s)\leq k_x(r)$, for $0<r\leq s$). Thanks to \cite[Remark $15.12$ (iii)]{S}, since $\nu \ne 0$, let's assume that $\lambda(S^{d-1})=1$, $\int_{(0,+\infty)} (r^2 \wedge 1)k_x(r)dr/r$ is finite and independent of $x$ and that $k_x(r)$ is right-continuous in $r>0$. Finally, since they satisfy the divergence condition (see e.g. \cite[Theorem 27.13]{S}), non-degenerate self-decomposable laws on $\mathbb{R}^d$ are absolutely continuous with respect to the $d$-dimensional Lebesgue measure. We further end this section by introducing some natural distances between probability measures on $\mathbb{R}^d$. For $p\geq 1$, the Wasserstein-$p$ distance between two probability measures $\mu_X$ and $\mu_Y$ with finite $p$-th moment is
\begin{align}
W_p(\mu_X,\mu_Y)=\underset{\Pi\in \Gamma(\mu_X, \mu_Y)}{\inf}\left(\int_{\mathbb{R}^d\times \mathbb{R}^d}\|x-y\|^pd\Pi(x,y)\right)^{\frac{1}{p}},
\end{align} 
where $\Gamma(\mu_X, \mu_Y)$ is the collection of probability measures on $\mathbb{R}^d\times \mathbb{R}^d$ with respective first and last $d$-dimensional marginals given by $\mu_X$ and $\mu_Y$. By H\"older inequality, for $1\leq p\leq q$,
\begin{align}
W_p(\mu_X,\mu_Y)\leq W_q(\mu_X,\mu_Y),
\end{align}
while, by duality,
\begin{align}
W_1(\mu_X,\mu_Y)=\underset{\|h\|_{Lip}\leq 1}{\sup}\left|\bbe h(X)-\bbe h(Y)\right|,
\end{align}
with $X\sim\mu_X$, $Y\sim\mu_Y$ and where $Lip$ is the space of Lipschitz functions on $\mathbb{R}^d$ endowed with the seminorm
\begin{align}
\|h\|_{Lip}=\underset{x\ne y\in \mathbb{R}^d}{\sup}\dfrac{\left| h(x)-h(y)\right|}{\|x-y \|}.
\end{align}
Let $\mathbb{N}^d$ be the space of multi-indices of dimension $d$.  For any $\alpha\in \mathbb{N}^d$, $|\alpha|=\sum_{i=1}^d |\alpha_i|$ and $D^{\alpha}$ the partial derivatives operators defined on smooth enough functions $f$, by $D^{\alpha}(f)(x_1,...,x_d)=\partial^{\alpha_1}_{x_1}...\partial^{\alpha_d}_{x_d}(f)(x_1,...,x_d)$, for all $(x_1,...,x_d)\in \mathbb{R}^d$. Moreover, for any function $r$-times continuously differentiable, $h$, on $\mathbb{R}^d$, viewing its $\ell$th-derivative $\mathbf{D}^{\ell}(h)$ as a $\ell$-multilinear form, for $1\leq \ell\leq r$, we introduce the following quantity
\begin{align}
M_{\ell}(h):=\underset{x\in \mathbb{R}^d}{\sup} \|\mathbf{D}^{\ell}(h)(x)\|_{op}=\underset{x\ne y}{\sup}\dfrac{\|\mathbf{D}^{\ell-1}(h)(x)-\mathbf{D}^{\ell-1}(h)(y)\|_{op}}{\|x-y\|}.
\end{align}
For $r\geq 0$, $\mathcal{H}_r$ is the space of bounded continuous functions defined on $\mathbb{R}^d$ which are continuously differentiable up to (and including) the order $r$ and such that, for any such function $f$ \begin{align}
\max_{0\leq \ell \leq r}M_\ell(f)\leq 1
\end{align}
with $M_0(f):=\sup_{x\in\mathbb{R}^d}|f(x)|$. In particular, for $f\in \mathcal{H}_r$,
\begin{align}
\underset{\alpha\in \mathbb{N}^d,\, 0\leq |\alpha|\leq r}{\max}\|D^\alpha(f)\|_{\infty}\leq 1.
\end{align}
Therefore, the space $\mathcal{H}_r$ is a subspace of the set of bounded functions which are $r$-times continuously differentiable on $\mathbb{R}^d$ such that $\|D^\alpha(f)\|_{\infty}\leq 1$ for all $\alpha\in \mathbb{N}^d$ with $0\leq |\alpha|\leq r$. Then, the smooth Wasserstein distance of order $r$ between two random vectors $X$ and $Y$ having respective laws $\mu_X$ and $\mu_Y$ is defined by
\begin{align}
d_{W_r}(\mu_X,\mu_Y)=\underset{h\in \mathcal{H}_r}{\sup} \left|\bbe h(X)-\bbe h(Y)\right|.
\end{align}
Moreover, the smooth Wasserstein distances of order $r\geq 1$ admit the following representation (see Lemma \ref{lem:repsmooth0} of the Appendix)
\begin{align}
d_{W_r}(\mu_X,\mu_Y)=\underset{h\in \mathcal{H}_r\cap C^{\infty}_c(\mathbb{R}^d)}{\sup} \left|\bbe h(X)-\bbe h(Y)\right|,
\end{align}
where $\mathcal{C}^{\infty}_c(\mathbb{R}^d)$ is the space of infinitely differentiable compactly supported functions on $\mathbb{R}^d$. In particular, for $p\geq 1$ and $r\geq 1$
\begin{align}\label{ineq:wasser}
d_{W_r}(\mu_X,\mu_Y) \leq d_{W_1}(\mu_X,\mu_Y)\leq W_1(\mu_X,\mu_Y)\leq W_p(\mu_X,\mu_Y).
\end{align}
Finally, as usual, for two probability measures, $\mu_1$ and $\mu_2$, on $\mathbb{R}^d$, $\mu_1$ is said to be absolutely continuous with respect to $\mu_2$, denoted by $\mu_1<<\mu_2$, if for any Borel set, $B$,  such that $\mu_2(B)=0$, then $\mu_1(B)=0$.

\section{Stein's Equation for SD Laws by Semigroup Methods}\label{sec:SteinEqu}
Let $X$ be a non-degenerate self-decomposable random vector with values in $\mathbb{R}^d$, without Gaussian component, and law $\mu_X$. By non-degenerate, we mean that the support of the law of $X$ is not contained in some $d-1$ dimensional subspace of $\mathbb{R}^d$. Denote by $X_i$, for $i= 1,...,d$, its coordinates and assume that $\bbe |X_i|<\infty$, for all $i=1,...,d$. Its characteristic function $\varphi$ is given, for all $\xi \in \mathbb{R}^d$, by
\begin{align}\label{eq:lvKh}
\varphi(\xi)&=\exp\left(i \langle \xi; \bbe X\rangle+\int_{\mathbb{R}^d}\left(e^{i \langle\xi ;u\rangle}-1-i\langle u ;\xi\rangle\right)\nu(du)\right)\nonumber\\
&= \exp\left(i \langle \xi; \bbe X\rangle+\int_{S^{d-1}\times (0,+\infty)}\left(e^{i \langle\xi ; r x\rangle}-1-i\langle rx ;\xi\rangle\right)\dfrac{k_x(r)}{r}dr\lambda(dx)\right),
\end{align}
where $\nu$ is the L\'evy measure of $X$, while $k_x$ and $\lambda$ are given in \eqref{rep:sd}. Further, assume that, for any $0<a<b<+\infty$ the functions $k_x(\cdot)$ are such that
\begin{align}\label{cond:kx}
\sup_{x\in S^{d-1}}\sup_{r\in (a,b)}k_x(r)<+\infty.
\end{align}
Since the function $k_x(\cdot)$ is a non-increasing function in $r>0$, the previous condition boils down to, 
\begin{align*}
\sup_{x\in S^{d-1}}k_x(a^+)<+\infty,\quad a>0.
\end{align*}
where $k_x(a^+)=\lim_{r\rightarrow a^+}k_x(r)$, for all $x\in S^{d-1}$. In \eqref{cond:kx}, the supremum over $x$ in $S^{d-1}$ has to be understood as the $\lambda$-essential supremum of the function $k_x(r)$ in the $x$ variable. In the univariate case, $d=1$, the polar decomposition of the L\'evy measure $\nu$ boils down to $\nu(du)=k(u)du/|u|$ where $k$ is non-negative, non-decreasing on $(-\infty,0)$ and non-increasing on $(0,+\infty)$. Thus, the condition \eqref{cond:kx} is automatically satisfied for $d=1$. For $d\geq 2$, the polar decomposition of the L\'evy measure associated with a stable distribution of index $\alpha\in (1,2)$ is given by 
\begin{align*}
\nu(du)=\bbone_{(0,+\infty)}(r)\bbone_{S^{d-1}}(x)\frac{dr}{r^{\alpha+1}}\lambda(dx),
\end{align*}
for some finite positive measure $\lambda$ on the $d$-dimensional unit sphere (see \cite[Theorem $14.3$]{S}). Then, the function $k_x(r)=1/r^\alpha$, for all $r>0$, and condition \eqref{cond:kx} is automatically satisfied (see below for more examples). Next, define a family of operators $(P^\nu_t)_{t\geq 0}$, for all $f\in{\cal S}(\mathbb{R}^d)$, all $x\in\mathbb{R}^d$ and all $t\geq 0$, via
\begin{align}\label{def:sg}
P^\nu_t(f)(x)=\dfrac{1}{(2\pi)^d}   \int_{\mathbb{R}^d} \mathcal{F}(f)(\xi) e^{i e^{-t}\langle x;\xi \rangle}\dfrac{\varphi(\xi)}{\varphi(e^{-t}\xi)} d\xi.
\end{align}
Denoting by $(\mu_t)_{t\geq 0}$ the family of probability measures given by \eqref{eq:2.1} with $\gamma=e^{-t}$ and using Fourier inversion in ${\cal S}(\mathbb{R}^d)$, then
\begin{align}\label{eq:2.2}
P^\nu_t(f)(x)= \int_{\mathbb{R}^d} f(u+e^{-t}x)\mu_t(du).
\end{align}
For all $t\geq 0$, the probability measure $\mu_t$ is infinitely divisible with finite first moment and its characteristic function $\varphi_t$ admits, for all $\xi \in \mathbb{R}^d$, the following representation
\begin{align}\label{rep:charac2}
\varphi_t(\xi)=\exp\left(i\langle\xi; \bbe X \rangle(1-e^{-t})+\int_{S^{d-1}\times (0,+\infty)}\left(e^{i \langle\xi ; r x\rangle}-1-i\langle rx ;\xi\rangle\right) \dfrac{k_x(r)-k_x(e^t r)}{r}dr\lambda(dx)\right).
\end{align}
The next lemma asserts that the family of operators $(P^\nu_t)_{t\geq 0}$ is a $C_0$-semigroup on the space $L^1(\mu_X)$. Its proof is very similar to the one dimensional case (see \cite[Proposition 5.1]{AH18}) thanks to the polar decomposition \eqref{rep:sd}.

\begin{lem}\label{lem:SG}
Let $X$ be a non-degenerate self-decomposable random vector without Gaussian component, with law $\mu_X$, L\'evy measure $\nu$ and such that $\bbe \|X\|<\infty$, with moreover the functions $k_x$ given by \eqref{rep:sd} satisfying \eqref{cond:kx}. Let $\varphi$ be its characteristic function and let $(P^\nu_t)_{t\geq 0}$ be the family of operators defined by \eqref{def:sg}. Then, $(P^\nu_t)_{t\geq 0}$ is a $C_0$-semigroup on the space $L^1(\mu_X)$ and its generator $\mathcal{A}$ is defined, for all $f\in {\cal S}(\mathbb{R}^d)$ and for all $x\in \mathbb{R}^d$, by
\begin{align}\label{def:Gene}
\mathcal{A}(f)(x)=\langle \bbe X-x;\nabla (f)(x) \rangle+\int_{\mathbb{R}^d}\langle \nabla(f)(x+u)-\nabla(f)(x) ;u\rangle \nu(du).
\end{align}
\end{lem}

\begin{proof}
Let $f\in\mathcal{C}_b(\mathbb{R}^d)$. First, by \eqref{eq:2.2}, for all $s,t\geq 0$ and for all $x\in\mathbb{R}^d$,
\begin{align*}
P^\nu_{s+t}(f)(x)=\int_{\mathbb{R}^d} f(u+e^{-(s+t)}x)\mu_{t+s}(du).
\end{align*}
Moreover, for all $s,t\geq 0$,
\begin{align}\label{eq:sem1}
P^\nu_t\circ P^\nu_s(f)(x)&=\int_{\mathbb{R}^d} P^\nu_s(f)(u+e^{-t}x)\mu_t(du)\nonumber\\
&=\int_{\mathbb{R}^d\times\mathbb{R}^d}f(v+e^{-s}(u+e^{-t}x))\mu_t(du)\mu_s(dv).
\end{align}
Let $\psi_{s,t,x}$ be the measurable function defined by $\psi_{s,t,x}(u,v)=v+e^{-s}(u+e^{-t}x)$, for all $u,v\in \mathbb{R}^d\times\mathbb{R}^d$. Then, from \eqref{eq:sem1}, for all $s,t\geq 0$ and for all $x\in \mathbb{R}^d$,
\begin{align*}
P^\nu_t\circ P^\nu_s(f)(x)=\int_{\mathbb{R}^d\times\mathbb{R}^d}f(w)(\mu_t\otimes \mu_s)\circ \psi_{s,t,x}^{-1}(dw).
\end{align*}
Let us now compute the characteristic function of the probability measure $(\mu_t\otimes \mu_s)\circ \psi_{s,t,x}^{-1}$. For all $\xi \in \mathbb{R}^d$,
\begin{align*}
\int_{\mathbb{R}^d}e^{i\langle \xi;w\rangle}(\mu_t\otimes \mu_s)\circ \psi_{s,t,x}^{-1}(dw)&=\int_{\mathbb{R}^d\times \mathbb{R}^d}e^{i \langle\xi ; \psi_{s,t,x}(u,v)\rangle}\mu_t(du)\otimes \mu_s(dv)\\
&=e^{i \langle \xi; e^{-(s+t)}x \rangle}\int_{\mathbb{R}^d} e^{i\langle \xi ;v\rangle}e^{i \langle \xi ; e^{-s}u\rangle}\mu_t(du)\otimes \mu_s(dv)\\
&=e^{i \langle \xi; e^{-(s+t)}x \rangle} \varphi_s(\xi)\varphi_{t}(e^{-s}\xi)\\
&=e^{i \langle \xi; e^{-(s+t)}x \rangle} \dfrac{\varphi(\xi)}{\varphi(e^{-(s+t)}\xi)}\\
&=e^{i \langle \xi; e^{-(s+t)}x \rangle} \int_{\mathbb{R}^d} e^{i \langle \xi; u \rangle}\mu_{s+t}(du)\\
&=\int_{\mathbb{R}^d} e^{i\langle \xi; \varphi_{s,t,x}(u) \rangle}\mu_{s+t}(du)=\int_{\mathbb{R}^d} e^{i\langle \xi; u \rangle}\mu_{s+t}\circ \varphi_{s,t,x}^{-1}(du),
\end{align*}
where $\varphi_{s,t,x}(u)=e^{-(s+t)}x+u$, for all $x\in \mathbb{R}^d$, $s\geq 0$ and $t\geq 0$. This implies that
\begin{align*}
P^\nu_t\circ P^\nu_s(f)(x)=P^\nu_{s+t}(f)(x),
\end{align*}
and so the semigroup property is verified on $\mathcal{C}_b(\mathbb{R}^d)$. Moreover, 
\begin{align}\label{eq:inv1}
\int_{\mathbb{R}^d} P_t^\nu(f)(x)\mu_X(dx)=\int_{\mathbb{R}^d\times\mathbb{R}^d}f(u+e^{-t}x)\mu_t(du)\mu_X(dx).
\end{align}
Setting $\omega_{t}(u,x)=u+e^{-t}x$, for all $u\in \mathbb{R}^d$ and all $x\in\mathbb{R}^d$. \eqref{eq:inv1} then becomes
\begin{align}
\int_{\mathbb{R}^d} P_t^\nu(f)(x)\mu_X(dx)=\int_{\mathbb{R}^d}f(v)(\mu_t \otimes \mu_X)\circ \omega_{t}^{-1}(dv).
\end{align}
The characteristic function of the probability measure $(d\mu_t \otimes d\mu_X)\circ \omega_{t}^{-1}$ is, for all $\xi \in \mathbb{R}^d$,
\begin{align*}
\int_{\mathbb{R}^d}e^{i\langle \xi ;v \rangle}(\mu_t \otimes \mu_X)\circ \omega_{t}^{-1}(dv)&=\int_{\mathbb{R}^d\times \mathbb{R}^d}e^{i\langle \xi ;\omega_{t}(u,x) \rangle}\mu_t(du)\mu_X(dx)\\
&=\int_{\mathbb{R}^d\times \mathbb{R}^d} e^{i\langle \xi ;u+e^{-t}x \rangle}\mu_t(du)\mu_X(dx)\\
&=\varphi_t(\xi)\varphi(e^{-t}\xi)=\varphi(\xi).
\end{align*}
Hence, for all $t\in \mathbb{R}^d$,
\begin{align}
\int_{\mathbb{R}^d} P_t^\nu(f)(x)\mu_X(dx)=\int_{\mathbb{R}^d} f(x)\mu_X(dx),
\end{align}
and so the probability measure $\mu_X$ is invariant for the family of operators $(P_t^\nu)_{t\geq 0}$. One can further check as well, by Fourier arguments, that, for all $x\in \mathbb{R}^d$,
\begin{align}
\underset{t\rightarrow 0^+}{\lim} P^\nu_t(f)(x)=f(x),\quad\quad \underset{t\rightarrow +\infty}{\lim} P^\nu_t(f)(x)=\int_{\mathbb{R}^d} f(x)\mu_X(dx).
\end{align}
Finally, by Jensen inequality and the invariance property,
\begin{align*}
\int_{\mathbb{R}^d}|P_t^\nu(f)(x)|\mu_X(dx)&\leq \int_{\mathbb{R}^d}P^\nu_t(|f|)(x)\mu_X(dx)\\
&\leq \int_{\mathbb{R}^d}|f(x)|\mu_X(dx).
\end{align*}
Then, by the density of $\mathcal{C}_b(\mathbb{R}^d)$ in $L^1(\mu_X)$ (\cite[Corollary $4.2.2$]{Bo07}), we can extend the family of  operators $(P_t^\nu)_{t\geq 0}$ to functions in $L^1(\mu_X)$. Moreover, this extension still denoted again by $(P_t^\nu)_{t\geq 0}$ is a $C_0$-semigroup on $L^1(\mu_X)$. To end the proof of the lemma, let us compute the generator of this semigroup on $\mathcal{S}(\mathbb{R}^d)$. Let $f\in \mathcal{S}(\mathbb{R}^d)$. By Fourier inversion, for all $x\in \mathbb{R}^d$ and for all $t>0$,
\begin{align*}
\frac{1}{t}\left(P_t^\nu(f)(x)-f(x)\right)=\frac{1}{(2\pi)^dt}\int_{\mathbb{R}^d} \mathcal{F}(f)(\xi)e^{i \langle \xi ;x\rangle}\left(e^{i \langle \xi ;x \rangle (e^{-t}-1)}\frac{\varphi(\xi)}{\varphi(e^{-t} \xi)}-1\right)d\xi.
\end{align*}
First, for all $x\in\mathbb{R}^d$ and for all $\xi \in \mathbb{R}^d$,
\begin{align*}
\underset{t\rightarrow 0^+}{\lim}\frac{1}{t}\left(e^{i \langle \xi ;x \rangle (e^{-t}-1)}\frac{\varphi(\xi)}{\varphi(e^{-t} \xi)}-1\right)=-i \langle \xi;x \rangle+\sum_{i=1}^d \xi_i\dfrac{\partial_i(\varphi)(\xi)}{\varphi(\xi)}
\end{align*}
which is a well-defined limit since $X$ has finite first moment. Now, from the representation \eqref{eq:lvKh}, for all $\xi \in \mathbb{R}^d$
\begin{align*}
\sum_{j=1}^d \xi_j\dfrac{\partial_j(\varphi)(\xi)}{\varphi(\xi)}&=\sum_{j=1}^d \xi_j\left(i\bbe X_j+i\int_{\mathbb{R}^d}u_j \left(e^{i\langle u;\xi \rangle}-1\right)\nu(du)\right)\\
&= i \left(\langle \xi; \bbe X \rangle+\int_{\mathbb{R}^d} \langle \xi; u \rangle\left(e^{i\langle u;\xi \rangle}-1\right)\nu(du)\right).
\end{align*}
Moreover, by Lemma \ref{lem:MomBounds} (ii) of the Appendix, for all $t\in (0,1)$, $x\in \mathbb{R}^d$ and $\xi \in \mathbb{R}^d$
\begin{align*}
\frac{1}{t}\left| e^{i\langle \xi ;x \rangle (e^{-t}-1)}\frac{\varphi(\xi)}{\varphi(e^{-t} \xi)}-1\right|\leq C_1 \|\xi\|\|x\|+C_2 (\|\xi\|+\|\xi\|\|\bbe X\|+\|\xi\|^2),
\end{align*}
for some $C_1,C_2>0$ independent of $t$, $\xi$ and $x$. Then, by the dominated convergence theorem, 
\begin{align*}
\underset{t\rightarrow 0^+}{\lim}\frac{1}{t}\left(P_t^\nu(f)(x)-f(x)\right)=\frac{1}{(2\pi)^d}\int_{\mathbb{R}^d} \mathcal{F}(f)(\xi)e^{i \langle \xi ;x\rangle}\bigg(&-i \langle \xi;x \rangle+i \langle \xi; \bbe X \rangle\\
&+i\int_{\mathbb{R}^d} \langle \xi; u \rangle\left(e^{i\langle u;\xi \rangle}-1\right)\nu(du) \bigg)d\xi,
\end{align*}
which is equal, by standard Fourier arguments, to
\begin{align*}
\underset{t\rightarrow 0^+}{\lim}\frac{1}{t}\left(P_t^\nu(f)(x)-f(x)\right)=\langle \nabla(f)(x); \bbe X-x \rangle+\int_{\mathbb{R}^d}\langle \nabla(f)(x+u)-\nabla(f)(x) ;u\rangle \nu(du).
\end{align*}
This concludes the proof of \eqref{def:Gene} and of the lemma.
\end{proof}
\noindent
Let $h \in \mathcal{H}_r\cap \mathcal{C}^{\infty}_c(\mathbb{R}^d)$, for some $r\geq 1$. The aim of this section is to solve, for all $x\in \mathbb{R}^d$, the following integro-partial differential equation,
\begin{align}\label{eq:Stein}
\langle \bbe X-x;\nabla (f)(x) \rangle+\int_{\mathbb{R}^d}\langle \nabla(f)(x+u)-\nabla(f)(x) ;u\rangle \nu(du)=h(x)-\bbe h(X),
\end{align}
which will serve as the fundamental equation in our Stein's methodology for non-degenerate self-decomposable law. As done in the one-dimensional case in \cite{AH18}, we first introduce a potential candidate solution for this equation, then study its regularity and finally prove that it actually solves the equation \eqref{eq:Stein}. The following proposition is concerned with the existence and the regularity of the candidate solution.

\begin{prop}\label{prop:SteinSol1}
Let $X$ be a non-degenerate self-decomposable random vector in $\mathbb{R}^d$ without Gaussian component, with law $\mu_X$, characteristic function $\varphi$ and such that $\bbe \|X\|<\infty$, with moreover the functions $k_x$ given by \eqref{rep:sd} satisfying \eqref{cond:kx}.
Let $(P^\nu_t)_{t\geq 0}$ be the semigroup of operators as in Lemma \ref{lem:SG}. Then, for any $h\in \mathcal{H}_2$, the function $f_h$ given, for all $x\in \mathbb{R}^d$, by
\begin{align}\label{eq:solStein1}
f_h(x)=-\int_{0}^{+\infty} (P^\nu_t(h)(x)-\bbe h(X))dt,
\end{align}
is well defined and twice continuously differentiable function on $\mathbb{R}^d$. Moreover, for any $\alpha\in \mathbb{N}^d$ such that $|\alpha|=1$,
\begin{align}
\| D^\alpha (f_h) \|_\infty \leq 1,
\end{align}
and, for any $\alpha\in \mathbb{N}^d$ such that $|\alpha|=2$,
\begin{align}
\| D^\alpha (f_h) \|_\infty \leq \frac{1}{2}.
\end{align}
\end{prop}

\begin{proof}
Let $h\in \mathcal{H}_2$. By \eqref{eq:2.2} and Theorem \ref{thm:smoothing} of the Appendix, for all $x\in \mathbb{R}^d$,
\begin{align*}
\left| P^\nu_t(h)(x)-\bbe h(X)\right|&\leq  \|x\| e^{-t} M_1(h) +d_{W_2}(X,X_t)\\
&\leq \|x\| e^{-t} +C_d e^{-\frac{t}{2^{d+1}(d+1)}},
\end{align*}
and so the function
\begin{align*}
f_h(x)=-\int_{0}^{+\infty} (P^\nu_t(h)(x)-\bbe h(X))dt,
\end{align*}
is well defined for all $x\in \mathbb{R}^d$. Moreover, by \eqref{eq:2.2} and the regularity of $h \in \mathcal{H}_2$, for all $1\leq j\leq d$ and for all $x\in \mathbb{R}^d$,
\begin{align*}
\partial_j(f_h)(x)=-\int_{0}^{+\infty} e^{-t}(P^\nu_t(\partial_j(h))(x))dt,
\end{align*}
which implies that, for all $1\leq j\leq d$ and for all $x\in \mathbb{R}^d$,
\begin{align*}
\left| \partial_j(f_h)(x)\right| \leq 1.
\end{align*}
Similarly for all $i,j\in \{1,..,d\}$ and for all $x\in\mathbb{R}^d$,
\begin{align*}
\partial^2_{ij}(f_h)(x)=-\int_{0}^{+\infty} e^{-2t}(P^\nu_t(\partial^2_{ij}(h))(x))dt,
\end{align*}
which implies that
\begin{align*}
\left|\partial^2_{ij}(f_h)(x)\right|\leq \frac{1}{2},
\end{align*}
and concludes the proof of the proposition.
\end{proof}

\begin{prop}\label{prop:SteinSol2}
Let $X$ be a non-degenerate self-decomposable random vector in $\mathbb{R}^d$ without Gaussian component, with law $\mu_X$, characteristic function $\varphi$, and such that $\bbe \|X\|<\infty$ with moreover the function $k_x$ given by \eqref{rep:sd} satisfying \eqref{cond:kx}. Further, let $(X_t)_{t\geq 0}$ be the collection of random vectors such that, for all $t\geq 0$, $X_t$ has law $\mu_t$ given by \eqref{eq:2.1} with $\gamma=e^{-t}$. For each $t>0$, let $\mu_t$ be absolutely continuous with respect to the $d$-dimensional Lebesgue measure and let its Radon-Nikodym derivative, denoted by $q_t$, be continuously differentiable on $\mathbb{R}^d$ and such that, for all $1\leq i\leq d$,
\begin{align}\label{HP:denmut}
\underset{y_i\rightarrow \pm \infty}{\lim} q_t(y)=0,\quad \int_{\mathbb{R}^d} \left|\partial_i(q_t)(y)\right|dy<\infty,\quad \int_0^{+\infty} e^{-2t} \left(\int_{\mathbb{R}^d} \left|\partial_i(q_t)(y)\right|dy\right) dt<\infty.
\end{align}
Let $h\in \mathcal{H}_1$ and $(P^\nu_t)_{t\geq 0}$ be the semigroup of operators as in Lemma \ref{lem:SG}. Then, the function $f_h$ given, for all $x\in \mathbb{R}^d$, by
\begin{align}
f_h(x)=-\int_{0}^{+\infty} (P^\nu_t(h)(x)-\bbe h(X))dt,
\end{align}
is a well defined twice continuously differentiable function on $\mathbb{R}^d$. Moreover, for any $\alpha\in \mathbb{N}^d$ such that $|\alpha|=1$
\begin{align}
\| D^\alpha (f_h) \|_\infty \leq 1,
\end{align}
and, for any $\alpha\in \mathbb{N}^d$ such that $|\alpha|=2$
\begin{align}
\| D^\alpha (f_h) \|_\infty \leq C_d.
\end{align}
for some $C_d>0$ only depending on $d$.
\end{prop}

\begin{proof}
Let $h\in \mathcal{H}_1$. By By \eqref{eq:2.2} and Theorem \ref{thm:smoothing}, for all $x\in \mathbb{R}^d$
\begin{align*}
\left| P^\nu_t(h)(x)-\bbe h(X)\right|&\leq \|x\| e^{-t}M_1(h) +d_{W_1}(X,X_t)\\
&\leq  \|x\| e^{-t} +C_d e^{-\frac{t}{2^{d+1}(d+1)}}
\end{align*}
and so the function
\begin{align*}
f_h(x)=-\int_{0}^{+\infty} (P^\nu_t(h)(x)-\bbe h(X))dt,
\end{align*}
is well defined for all $x\in \mathbb{R}^d$. Moreover, by \eqref{eq:2.2} and the regularity of $h\in \mathcal{H}_1$, for all $1\leq j\leq d$ and for all $x\in \mathbb{R}^d$,
\begin{align*}
\partial_j(f_h)(x)=-\int_{0}^{+\infty} e^{-t}(P^\nu_t(\partial_j(h))(x))dt,
\end{align*}
which implies that, for all $1\leq j\leq d$ and for all $x\in \mathbb{R}^d$
\begin{align*}
\left| \partial_j(f_h)(x)\right| \leq 1.
\end{align*}
Let us fix $1\leq i\leq d$ and $x\in \mathbb{R}^d$. By definition and an integration by parts (thanks to \eqref{HP:denmut}),
\begin{align*}
P^\nu_t(\partial_i(h))(x)&= \int_{\mathbb{R}^d}\partial_i(h)(xe^{-t}+y) q_t(y)dy\\
&=-\int_{\mathbb{R}^d} h(xe^{-t}+y)\partial_i(q_t)(y)dy.
\end{align*}
Thus, for all $1\leq i,j\leq d$ and $x\in \mathbb{R}^d$,
\begin{align*}
\partial_j\left(P^\nu_t(\partial_i(h))(x)\right)&=e^{-t}\int_{\mathbb{R}^d}\partial_j(h)(xe^{-t}+y)\partial_i(q_t)(y)dy.
\end{align*}
This representation together with the third condition in \eqref{HP:denmut} ensures that $f_h$ is twice continuously differentiable on $\mathbb{R}^d$ and that, for all $x\in \mathbb{R}^d$ and for all $1\leq i,j\leq d$
\begin{align*}
\partial_{i,j}(f_h)(x)=-\int_0^{+\infty}e^{-2t}\left(\int_{\mathbb{R}^d}\partial_j(h)(xe^{-t}+y)\partial_i(q_t)(y)dy\right)dt.
\end{align*}
Finally, for all $x\in \mathbb{R}^d$ and all $1\leq i,j\leq d$,
\begin{align*}
\left|\partial_{i,j}(f_h)(x)\right|\leq \int_0^{+\infty}e^{-2t}\left(\int_{\mathbb{R}^d}\left|\partial_i(q_t)(y)\right|dy\right)dt<+\infty.
\end{align*}
This concludes the proof of the proposition.
\end{proof}
\noindent
Before showing that $f_h$ as given in the two previous propositions
is a solution to the integro-partial differential equation
\begin{align*} 
\langle \bbe X-x;\nabla (f)(x) \rangle+\int_{\mathbb{R}^d}\langle \nabla(f)(x+u)-\nabla(f)(x) ;u\rangle \nu(du)=h(x)-\bbe h(X),\quad\quad x\in \mathbb{R}^d,
\end{align*}
for $h$ respectively in $\mathcal{H}_2 \cap \mathcal{C}^{\infty}_c(\mathbb{R}^d)$ and $\mathcal{H}_1 \cap \mathcal{C}^{\infty}_c(\mathbb{R}^d)$, we provide some examples of non-degenerate self-decomposable random vectors whose law $\mu_X$ satisfies the assumptions of the aforementioned propositions.\\
\\
\textbf{\large Some Examples}\\
\textbf{Rotationally invariant $\alpha$-stable random vector in $\mathbb{R}^d$}.
Let $\alpha\in (1,2)$ and let $X$ be an $\alpha$-stable random vector whose law is rotationally invariant. Then, its characteristic function $\varphi$ is, for all $\xi\in \mathbb{R}^d$
\begin{align*}
\varphi(\xi)=\exp\left(-C_{\alpha,d}\|\xi\|^\alpha\right),
\end{align*}
for some constant $C_{\alpha,d}>0$ depending on $\alpha$ and $d$. Hence, the characteristic function of $\mu_t$ is given, for all $\xi\in \mathbb{R}^d$, and $t\geq 0$ by
\begin{align*}
\varphi_t(\xi)=\exp\left(-C_{\alpha,d}(1-e^{-\alpha t})\|\xi\|^\alpha\right).
\end{align*}
Thus, for all $t> 0$, $\mu_t$ is absolutely continuous with respect to the Lebesgue measure and its density $q_t$ is given for all $x\in \mathbb{R}^d$, by the Fourier inversion formula, 
\begin{align}\label{eq:FourRepDen}
q_t(x)=\frac{1}{(2\pi)^d}\int_{\mathbb{R}^d}e^{i \langle \xi; x \rangle}\overline{\varphi_t(\xi)}d\xi=\frac{1}{(2\pi)^d}\int_{\mathbb{R}^d}e^{i \langle \xi; x \rangle}\exp\left(-C_{\alpha,d}(1-e^{-\alpha t})\|\xi\|^\alpha\right)d\xi.
\end{align}
From \eqref{eq:FourRepDen}, it is clear that $q_t$ is continuously differentiable on $\mathbb{R}^d$ and that, for all $x\in \mathbb{R}^d$ and $1\leq j\leq d$,
\begin{align*}
\partial_j(q_t)(x)=\frac{1}{(2\pi)^d}\int_{\mathbb{R}^d}e^{i \langle \xi; x \rangle}(i\xi_j)\exp\left(-C_{\alpha,d}(1-e^{-\alpha t})\|\xi\|^\alpha\right)d\xi.
\end{align*}
Moreover, the characteristic function $\varphi_t$ is linked to the Fourier transform of the probability transition density of a rotationally invariant $d$-dimensional $\alpha$-stable process after the time change $\tau=(1-e^{-\alpha t})$. Indeed, if $(Z^\alpha_t)_{t\geq 0}$ is a rotationally invariant $d$-dimensional $\alpha$-stable L\'evy process, then its characteristic function at time $t$ is given, for all $t\geq 0$ and all $\xi\in \mathbb{R}^d$, by
\begin{align*}
\bbe\, e^{i \langle \xi ;Z^\alpha_t\rangle}=\exp\left(-\dfrac{t\|\xi\|^{\alpha}}{2^{\frac{\alpha}{2}}}\right).
\end{align*}
Thus, for all $\xi\in \mathbb{R}^d$ and all $t\geq 0$
\begin{align*}
\varphi_t(\xi)=\bbe\, e^{i \langle \xi ;\sqrt{2}C_{\alpha,d}^{\frac{1}{\alpha}}Z^\alpha_{1-e^{-\alpha t}}\rangle}.
\end{align*}
Finally, Lemma $2.2$ of \cite{CZ16} implies that the density $q_t$ and its gradient satisfy the following inequality, for all $x\in \mathbb{R}^d$
\begin{align*}
|q_t(x)| \leq C'_{\alpha,d} \dfrac{(1-e^{-\alpha t})}{\left((1-e^{-\alpha t})^{\frac{1}{\alpha}}+\|x\|\right)^{\alpha+d}},\quad
\|\nabla (q_t)(x)\|\leq C''_{\alpha,d} \dfrac{(1-e^{-\alpha t})}{\left((1-e^{-\alpha t})^{\frac{1}{\alpha}}+\|x\|\right)^{\alpha+d+1}},
\end{align*} 
for some $C'_{\alpha,d}, C''_{\alpha,d}>0$ only depending on $\alpha$ and $d$. It follows that $q_t$ satisfies the conditions \eqref{HP:denmut}.\\
\\
\textbf{Symmetric $\alpha$-stable random vector in $\mathbb{R}^d$}. Let $\alpha \in (1,2)$ and let $X$ be a symmetric $\alpha$-stable random vector on $\mathbb{R}^d$. By \cite[Theorem 14.13]{S}, the characteristic function of $X$ is given by, for all $\xi \in \mathbb{R}^d$ with $\|\xi\|\ne 0$
\begin{align*}
\varphi(\xi)=\exp\left(-\int_{S^{d-1}}\left|\langle x;\xi \rangle\right|^\alpha \lambda_1(dx)\right)=\exp\left(-\|\xi\|^\alpha\int_{S^{d-1}}\left|\left\langle x;\frac{\xi}{\|\xi\|} \right\rangle\right|^\alpha \lambda_1(dx)\right),
\end{align*}
where $\lambda_1$ is a symmetric positive finite measure on $S^{d-1}$. Then, for all $\xi \in \mathbb{R}^d$ and all $t\geq 0$,
\begin{align*}
\varphi_t(\xi)=\exp\left(-(1-e^{-\alpha t})\|\xi\|^\alpha\int_{S^{d-1}}\left|\left\langle x;\frac{\xi}{\|\xi\|} \right\rangle\right|^\alpha \lambda_1(dx)\right).
\end{align*}
Moreover, let's assume that there exists $c_0>0$ such that for any $u\in S^{d-1}$, $\int_{S^{d-1}}|\langle x;u \rangle|^\alpha \lambda_1(dx)\geq c_0$. Then, $\mu_t$ is absolutely continuous with respect to the $d$-dimensional Lebesgue measure and its density $q_t$ is given, for all $t> 0$ and all $x\in \mathbb{R}^d$, by
\begin{align*}
q_t(x)=\frac{1}{(2\pi)^d}\int_{\mathbb{R}^d}e^{i \langle \xi; x \rangle}\overline{\varphi_t(\xi)}d\xi=\frac{1}{(2\pi)^d}\int_{\mathbb{R}^d}e^{i \langle \xi; x \rangle}\exp\left(-(1-e^{-\alpha t})\|\xi\|^\alpha\eta_\alpha(\xi)\right)d\xi,
\end{align*}
with $\eta_\alpha(\xi)=\int_{S^{d-1}}|\langle x;\frac{\xi}{\|\xi\|} \rangle|^\alpha \lambda_1(dx)$. For all $t>0$, $q_t$ is continuously differentiable on $\mathbb{R}^d$ and its partial derivative in direction $j\in \{1,...,d\}$ is given by, for all $x\in \mathbb{R}^d$
\begin{align*}
\partial_j(q_t)(x)=\frac{1}{(2\pi)^d}\int_{\mathbb{R}^d}e^{i \langle \xi; x \rangle}(i\xi_j)\exp\left(-(1-e^{-\alpha t})\|\xi\|^\alpha\eta_\alpha(\xi)\right)d\xi.
\end{align*}
\\
\textbf{Takano distribution \cite{T89}}. Let $\alpha\in (0,+\infty)$ and let $\mu_{\alpha,d}$ be the probability measure on $\mathbb{R}^d$ given by
\begin{align*}
\mu_{\alpha,d}(dx)=c_{\alpha,d}\left(1+\|x\|^2\right)^{-\alpha-d/2}dx,
\end{align*}
where $c_{\alpha,d}>0$ is a normalizing constant. As shown in \cite{T89} such a probability measure is self-decomposable. Moreover, its characteristic function is given, for all $\xi\in \mathbb{R}^d$, by (\cite[Theorem II]{T89})
\begin{align*}
\varphi(\xi)=&\exp\bigg(\int_0^{+\infty}2(2\pi)^{-d/2}v^{d/2}\left(\int_0^{+\infty}(\sqrt{2w})^{(d+2)/2}K_{(d-2)/2}(\sqrt{2w}v)g_\alpha(2w)dw\right)dv\\ 
&\quad\quad\quad \times\int_{S^{d-1}}dx \int_0^v \left(e^{i u x \xi}-1-\dfrac{i u x\xi}{1+u^2}\right)\frac{du}{u}\bigg),
\end{align*}
where $K_{(d-2)/2}$ denotes the modified Bessel function of order $(d-2)/2$ while $g_\alpha(w)=2/(\pi^2 w )\\\times1/\left(J_\alpha^2(\sqrt{w})+Y_\alpha^2(\sqrt{w})\right))$, $w>0$, with $J_\alpha$ and $Y_\alpha$ the Bessel functions of the first kind and of the second kind, respectively (see \cite[Chapter $10$]{OLBC10} for definitions of Bessel functions). Finally, the L\'evy measure of $\mu_{\alpha,d}$ is given by (see \cite[Representation $(2)$ page $23$]{T89})
\begin{align*}
\nu(du)=\frac{2}{\|u\|^d} \left(\int_0^{+\infty}g_\alpha(2w)L_{d/2}\left(\sqrt{2w}\|u\|\right)dw\right) du
\end{align*}
where $L_{d/2}(v)=(2\pi)^{-d/2}v^{d/2}K_{d/2}(v)$, for all $v>0$. This implies the following polar decomposition for $\nu$
\begin{align*}
\nu(du)=\bbone_{(0,+\infty)}(r)\bbone_{S^{d-1}}(x)\frac{2}{r} \left(\int_0^{+\infty}g_\alpha(2w)L_{d/2}\left(\sqrt{2w}r\right)dw\right) dr\sigma(dx),
\end{align*}
where $\sigma$ is the uniform measure on $S^{d-1}$. Hence, condition \eqref{cond:kx} is automatically satisfied. So, our methodology applies as soon as $\alpha>1/2$ (which ensures that $\int_{\mathbb{R}^d}\|x\|\mu_{\alpha,d}(dx)<+\infty$).\\
\\
\textbf{Takano distribution \cite{T88}}: Let $\mu$ be the probability measure on $\mathbb{R}^d$ given by
\begin{align*}
\mu(dx)=C \exp\left(-\|x\|\right)dx,
\end{align*}
where $C>0$ is a normalizing constant. Thanks to \cite[Result $1$]{T88}, its characteristic function $\varphi$ is given by
\begin{align*}
\varphi(\xi)=\exp\left(\int_{\mathbb{R}^d}\left(e^{i \langle \xi;u \rangle}-1\right)\frac{M(\|u\|)}{\|u\|^d}du\right),\quad \xi\in \mathbb{R}^d,
\end{align*}
with $M(w)=(2\pi)^{-d/2}(d+1)w^{d/2}K_{d/2}(w)$, for $w>0$. Hence, $\mu$ is an infinitely divisible probability measure on $\mathbb{R}^d$. Moreover, the function $M$ admits the following representation (see the last formula page $64$ of \cite{T88})
\begin{align*}
M(w)=C_d \int_{w}^{+\infty} v^{d/2} K_{(d-2)/2}(v)dv, \quad w>0,
\end{align*}
which is non-negative and non-increasing on $(0,+\infty)$ (and $C_d>0$). Thus, $\mu$ is self-decomposable and its L\'evy measure admits the following polar decomposition
\begin{align*}
\nu(du)=\bbone_{(0,+\infty)}(r)\bbone_{S^{d-1}}(x)\frac{M(r)}{r}dr \sigma(dx),
\end{align*}
where $\sigma$ is the uniform measure on the Euclidean unit sphere. Finally, the associated $k_x$ functions satisfy \eqref{cond:kx} and the probability measure $\mu$ admits moments of any orders.\\
\\
\textbf{Multivariate gamma distributions}. Let $(\alpha_1,\dots,\alpha_d)\in (0,+\infty)^d$ and let $X=(X_1,\dots,X_d)$ be a random vector whose independent coordinates are distributed according to gamma laws with parameters $(\alpha_i,1)$, $1\leq i\leq d$. Namely, the characteristic function of $X$ is given by
\begin{align*}
\varphi(\xi)=\prod_{j=1}^d \left(1-i\xi_j\right)^{-\alpha_j},\quad \xi\in \mathbb{R}^d.
\end{align*} 
For any $b>1$, there exists $\rho_b$, a probability measure on $\mathbb{R}^d$, such that, $\varphi(\xi)=\varphi(\xi/b)\hat{\rho}_b(\xi)$, for all $\xi \in \mathbb{R}^d$. Indeed, take $\rho_b=\rho_{1,b}\otimes ...\otimes \rho_{d,b}$, where, for all $1\leq j\leq d$, $\rho_{j,b}$ are defined, for all $\xi_j\in \mathbb{R}$, by
\begin{align*}
\int_{\mathbb{R}}e^{i \xi_j x}{\rho}_{j,b}(dx)=\dfrac{\left(1-i\xi_j\right)^{-\alpha_j}}{\left(1-i\frac{\xi_j}{b}\right)^{-\alpha_j}}.
\end{align*}
Then, one can apply our methodology to these multivariate gamma distributions. Moreover, for all $\xi \in \mathbb{R}^d$ and all $t> 0$,
\begin{align*}
e^{-t\sum_{j=1}^d\alpha_j}\leq \left|\dfrac{\varphi(\xi)}{\varphi(e^{-t}\xi)}\right|\leq 1.
\end{align*}
Let us note that other types of self-decomposable multivariate gamma distributions have been considered in the literature (see \cite{PS14}). In particular, the authors of \cite{PS14} considered infinitely divisible multivariate gamma distributions whose L\'evy measures have the following polar decomposition,
\begin{align*}
\nu(du)=\bbone_{(0,+\infty)}(r)\bbone_{S^{d-1}}(x)\alpha\frac{\exp(-\beta r)}{r}dr\lambda(dx),
\end{align*}
where $\alpha, \beta$ are positive real numbers and $\lambda$ is a finite positive measure on the $d$-dimensional unit sphere. 
Therefore, $k_x(r)=\alpha e^{-\beta r}$, $r>0$ and so \eqref{cond:kx} is again satisfied. Moreover, such infinitely divisible multivariate gamma distributions are self-decomposable and admit absolute moment of any orders.\\
\\
To end this series of examples, let us mention one other way to build probability measures on $\mathbb{R}^d$ which are self-decomposable. One standard procedure is through mixtures. For example, thanks to \cite[Corollary p. 40]{T289}, the mixture of a $d$-multivariate normal distribution $\mathcal{N}(m, \Gamma I_d)$, $m\in \mathbb{R}^d$, and a generalized gamma convolution $\Gamma$ (see \cite{Bond92} for a definition) is self-decomposable.\\
\\
Let us next solve the integro-partial differential equation \eqref{eq:Stein} for any $h\in \mathcal{H}_2\cap \mathcal{C}^{\infty}_c(\mathbb{R}^d)$. Note that under the assumption of Proposition \ref{prop:SteinSol2}, it is possible to solve, \textit{mutatis mutandis}, the Stein equation \eqref{eq:Stein} for $h\in \mathcal{H}_1\cap \mathcal{C}^{\infty}_c(\mathbb{R}^d)$.

\begin{prop}\label{prop:SteinEq}
Let $X$ be a non-degenerate self-decomposable random vector in $\mathbb{R}^d$ without Gaussian component, with law $\mu_X$, with L\'evy measure $\nu$ and such that $\bbe \|X\|<\infty$, with moreover the functions $k_x$ given by \eqref{rep:sd} satisfying \eqref{cond:kx}. Let $h\in \mathcal{H}_2\cap \mathcal{C}^{\infty}_c(\mathbb{R}^d)$ and let $f_h$ be the function given by Proposition \ref{prop:SteinSol1}. Then, for all $x\in \mathbb{R}^d$
\begin{align*}
\langle \bbe X-x;\nabla (f_h)(x) \rangle+\int_{\mathbb{R}^d}\langle \nabla(f_h)(x+u)-\nabla(f_h)(x) ;u\rangle \nu(du)=h(x)-\bbe h(X).
\end{align*}
\end{prop}

\begin{proof}
Let $h\in \mathcal{H}_2\cap \mathcal{C}^{\infty}_c(\mathbb{R}^d)$ and $f_h$ be given by \eqref{eq:solStein1}. Let $\hat{h}=h-\bbe h(X)$.\\
\textit{Step 1}:
Let us prove that for all $t>0$ and all $x\in \mathbb{R}^d$
\begin{align}\label{eq:heat}
\dfrac{d}{dt} \left(P_t^\nu(h)(x)\right)= \mathcal{A}(P_t^\nu(h))(x).
\end{align}
Since $h$ belongs to $\mathcal{C}^{\infty}_c(\mathbb{R}^d)$, by Fourier inversion, for all $t> 0$ and for all $x\in \mathbb{R}^d$
\begin{align*}
P_t^\nu(h)(x)=\int_{\mathbb{R}^d} \mathcal{F}(h)(\xi) e^{i e^{-t}\langle x;\xi \rangle}\dfrac{\varphi(\xi)}{\varphi(e^{-t}\xi)} \dfrac{d\xi}{(2\pi)^d},
\end{align*}
which will be used to compute $d/dt (P_t^\nu(h)(x))$. First, note that, for all $x\in \mathbb{R}^d$, $\xi \in \mathbb{R}^d$ and $t> 0$,
\begin{align*}
\dfrac{d}{dt}\left(e^{i e^{-t}\langle x;\xi \rangle}\dfrac{\varphi(\xi)}{\varphi(e^{-t}\xi)}\right)&=-ie^{-t}\langle x;\xi \rangle e^{i e^{-t}\langle x;\xi \rangle}\dfrac{\varphi(\xi)}{\varphi(e^{-t}\xi)}+e^{i e^{-t}\langle x;\xi \rangle}\varphi(\xi)e^{-t}\left(\sum_{j=1}^d \xi_j\dfrac{\partial_j(\varphi)(e^{-t}\xi)}{\varphi(\xi e^{-t})^2}\right)\\
&=e^{i e^{-t}\langle x;\xi \rangle}\dfrac{\varphi(\xi)}{\varphi(e^{-t}\xi)}\left(-ie^{-t}\langle x;\xi \rangle+e^{-t}\sum_{j=1}^d \xi_j\dfrac{\partial_j(\varphi)(e^{-t}\xi)}{\varphi(\xi e^{-t})}\right)\\
&=e^{i e^{-t}\langle x;\xi \rangle}\dfrac{\varphi(\xi)}{\varphi(e^{-t}\xi)}ie^{-t}\left(\langle \bbe X-x;\xi \rangle+\int_{\mathbb{R}^d}\langle u;\xi \rangle(e^{i\langle u; e^{-t}\xi \rangle}-1)\nu(du)\right).
\end{align*}
Moreover, for all $x\in \mathbb{R}^d$, for all $\xi \in \mathbb{R}^d$ and all $t> 0$,
\begin{align*}
\left| \dfrac{d}{dt}\left(e^{i e^{-t}\langle x;\xi \rangle}\dfrac{\varphi(\xi)}{\varphi(e^{-t}\xi)}\right) \right| &\leq e^{-t} \left( \| \bbe X-x\| \|\xi\|+\|\xi\|^2 e^{-t}\int_{\|u\|\leq 1}\|u\|^2 \nu(du)+2\|\xi\| \int_{\|u\|\geq 1}\|u\|\nu(du)\right)\\
&\leq \left( \| \bbe X-x\| \|\xi\|+\|\xi\|^2 \int_{\mathbb{R}^d}\|u\|^2 \nu(du)+2\|\xi\| \int_{\|u\|\geq 1}\|u\|\nu(du)\right),
\end{align*}
hence, 
\begin{align*}
\dfrac{d}{dt} (P_t^\nu(h)(x))=\int_{\mathbb{R}^d} \mathcal{F}(h)(\xi)e^{i e^{-t}\langle x;\xi \rangle}\dfrac{\varphi(\xi)}{\varphi(e^{-t}\xi)}ie^{-t}\left(\langle \bbe X-x;\xi \rangle+\int_{\mathbb{R}^d}\langle u;\xi \rangle(e^{i\langle u; e^{-t}\xi \rangle}-1)\nu(du)\right)\frac{d\xi}{(2\pi)^d}.
\end{align*}
To conclude the first step, let us precisely compute the right-hand side of the previous equality. First,
\begin{align*}
\int_{\mathbb{R}^d} \mathcal{F}(h)(\xi)e^{i e^{-t}\langle x;\xi \rangle}\dfrac{\varphi(\xi)}{\varphi(e^{-t}\xi)}ie^{-t}\langle \bbe X-x;\xi \rangle \frac{d\xi}{(2\pi)^d} &=\langle \bbe X-x;e^{-t}P^\nu_t(\nabla (h))(x)\rangle\\
&=\langle \bbe X-x;\nabla (P^\nu_t(h))(x)\rangle.
\end{align*}
where we have used that $e^{-t}P^\nu_t(\partial_j (h))(x)=\partial_j (P^\nu_t(h))(x)$, for all $t\geq0$, all $x\in \mathbb{R}^d$ and all $1\leq j\leq d$. Moreover, by Fubini Theorem and Fourier arguments,
\begin{align*}
(I)&:=\int_{\mathbb{R}^d} \mathcal{F}(h)(\xi)e^{i e^{-t}\langle x;\xi \rangle}\dfrac{\varphi(\xi)}{\varphi(e^{-t}\xi)}ie^{-t}\left(\int_{\mathbb{R}^d}\langle u;\xi \rangle(e^{i\langle u; e^{-t}\xi \rangle}-1)\nu(du)\right)\frac{d\xi}{(2\pi)^d}\\
&=\sum_{j=1}^d \int_{\mathbb{R}^d} \bigg(\int_{\mathbb{R}^d}u_j \bigg(\mathcal{F}(\partial_j(h)(.+e^{-t}u))(\xi)-\mathcal{F}(\partial_j(h))(\xi)\bigg)\nu(du)\bigg)e^{i e^{-t}\langle x;\xi \rangle}\dfrac{\varphi(\xi)}{\varphi(e^{-t}\xi)}e^{-t}\frac{d\xi}{(2\pi)^d}\\
&=\sum_{j=1}^d \int_{\mathbb{R}^d} u_j e^{-t}\left(P^\nu_t(\partial_j(h))(x+u)-P^\nu_t(\partial_j(h))(x)\right)\nu(du)\\
&=\int_{\mathbb{R}^d} \langle u; \nabla(P^\nu_t(h))(x+u)-\nabla(P^\nu_t(h))(x) \rangle \nu(du).
\end{align*}
Thus, for all $x\in \mathbb{R}^d$ and all $t>0$, 
\begin{align*}
\dfrac{d}{dt}\left(P^\nu_t(h))\right)(x)=\mathcal{A}(P_t^\nu(h))(x),
\end{align*}
which gives \eqref{eq:heat} and finishes the proof of the first step.\\
\textit{Step 2}: Let $0<b<+\infty$. Integrating out the equality \eqref{eq:heat} gives,
\begin{align*}
P_b^\nu(h)(x)-h(x)=\int_{0}^b \mathcal{A}(P_t^\nu(h))(x) dt,
\end{align*}
then, letting $b\rightarrow +\infty$ and using Lemma \ref{lem:SG} lead to:
\begin{align*}
\lim_{b\rightarrow+\infty}\left(P_b^\nu(h)(x)-h(x)\right)=-\hat{h}(x),\quad\quad x\in \mathbb{R}^d.
\end{align*}
Next, let us show that $\int_{0}^{+\infty} \left|\mathcal{A}(P_t^\nu(h))(x)\right| dt<+\infty$, for all $x\in \mathbb{R}^d$. To do so, we need to estimate the quantities $\|\nabla (P^\nu_t(h))(x)\|$ and $\|\nabla(P_t^\nu(h))(x+u)-\nabla(P_t^\nu(h))(x)\|$, for all $x\in \mathbb{R}^d$, all $u\in \mathbb{R}^d$ and all $t\geq 0$. Using that $\partial_j\left(P^\nu_t(h)\right)(x)=e^{-t}P^\nu_t(\partial_j(h))(x)$ and that $h\in \mathcal{H}_2$,
\begin{align*}
\|\nabla (P^\nu_t(h))(x)\|\leq \sqrt{d}e^{-t}, \quad \|\nabla(P_t^\nu(h))(x+u)-\nabla(P_t^\nu(h))(x)\| \leq \sqrt{d} e^{-t} \bbone_{\|u\|\geq 1}+\sqrt{d}e^{-t} \|u\| \bbone_{\|u\|\leq 1}.
\end{align*}
Then, by the very definitions of $\cal A$ and $P^\nu_t$ and standard inequalities, for all $x\in \mathbb{R}^d$ and all $t\geq 0$
\begin{align*}
\left|\mathcal{A}(P_t^\nu(h))(x)\right| \leq \sqrt{d}e^{-t}\left(\|\bbe X-x\|+ \left(\int_{\|u\|\leq 1}\|u\|^2\nu(du)+ \int_{\|u\|\geq 1}\|u\|\nu(du)\right)\right),
\end{align*}
which implies that $\int_{0}^{+\infty} \left|\mathcal{A}(P_t^\nu(h))(x)\right| dt<+\infty$, for all $x\in \mathbb{R}^d$. Moreover, $\mathcal{A}(P_t^\nu(h))(x)=\mathcal{A}(P_t^\nu(\hat{h}))(x)$, thus, for all $x\in \mathbb{R}^d$,
\begin{align*}
-\hat{h}(x)=\int_0^{+\infty}\mathcal{A}(P_t^\nu(\hat{h}))(x) dt.
\end{align*}
To conclude, one needs to prove that, for all $x\in \mathbb{R}^d$
\begin{align*}
\int_0^{+\infty}\mathcal{A}(P_t^\nu(\hat{h}))(x) dt=\mathcal{A} \left(\int_0^{+\infty}P_t^\nu(\hat{h})(x)dt\right)=-\mathcal{A}(f_h)(x),
\end{align*}
but this follows from standard arguments as well as from Proposition \ref{prop:SteinSol1}.
\end{proof}
\noindent
We end this section with regularity estimates for the solution $f_h$ of the Stein's equation under the assumptions of Proposition \ref{prop:SteinSol1}. Similar estimates hold true under the assumptions of Proposition \ref{prop:SteinSol2}. In particular, under these latter assumptions, it is sufficient to have $h\in \mathcal{H}_1$ to obtain a bound on $M_2(f_h)$, and this is in line with the Gaussian case (see \cite{Rai04,ChM08}).

\begin{prop}\label{EstM}
Let $X$ be a non-degenerate self-decomposable random vector in $\mathbb{R}^d$ without Gaussian component, with law $\mu_X$, characteristic function $\varphi$ and such that $\bbe \|X\|<\infty$, with moreover the functions $k_x$ given by \eqref{rep:sd} satisfying \eqref{cond:kx}.
Let $h\in\mathcal{H}_2$ and $(P^\nu_t)_{t\geq 0}$ be the semigroup of operators as in Lemma \ref{lem:SG}. Then, $f_h$ given, for all $x\in \mathbb{R}^d$, by
\begin{align*}
f_h(x)=-\int_{0}^{+\infty} (P^\nu_t(h)(x)-\bbe h(X))dt,
\end{align*}
is such that
\begin{align*}
M_1(f_h)\leq \sqrt{d},\quad\quad M_2(f_h)\leq \frac{d}{2}.
\end{align*}
\end{prop}

\begin{proof}
By definition,
\begin{align*}
M_1(f_h)=\sup_{x\in \mathbb{R}^d} \|\nabla(f_h)(x)\|_{op}.
\end{align*}
Let $u\in \mathbb{R}^d$ with $\|u\|=1$. Then, by the Cauchy-Schwarz inequality, for all $x\in \mathbb{R}^d$
\begin{align*}
|\langle\nabla(f_h)(x);u \rangle| \leq \|\nabla(f_h)(x)\|.
\end{align*}
Now, thanks to the commutation relation $\partial_i\left(P^\nu_t(h)\right)(x)=e^{-t}P^\nu_t(\partial_i(h))(x)$ and since $h\in \mathcal{H}_2$, for all $1\leq i\leq d$ and all $x\in \mathbb{R}^d$
\begin{align*}
|\partial_i(f_h)(x)| \leq 1,
\end{align*}
implying that $M_1(f_h)\leq \sqrt{d}$. The bound for $M_2(f_h)$ follows similarly using the commutation relation twice and the fact that $h\in \mathcal{H}_2$.
\end{proof}


\section{Stein Kernels for SD Laws With Finite Second Moment}\label{sec:SK}
In the Gaussian setting, a major finding in the context of Stein's method is the introduction of the notion of Stein kernel (see e.g. \cite{Stein2,CPU94,Cha08,Cha09,NP1,Cha12,NP3,NPS14,LNP15,CFP18}). Recall that $\gamma$, a centered Gaussian measure on $\mathbb{R}^d$, satisfies the following integration by parts formula,
\begin{align*}
\int_{\mathbb{R}^d} \langle x ; f(x)\rangle \gamma(dx) =\int_{\mathbb{R}^d} \operatorname{div} (f(x))\gamma(dx),
\end{align*}
for all smooth enough, $\mathbb{R}^d$-valued function $f=(f_1,...,f_d)$ and where $\operatorname{div}(f(x))=\sum_{j=1}^d \partial_j(f_j)(x)$. For a centered probability measure $\rho$ on $\mathbb{R}^d$, the Gaussian Stein kernel of $\rho$ is the measurable function $\tau_\rho$, from $\mathbb{R}^d$ to $\mathcal{M}_{d\times d}(\mathbb{R})$, the space of $d\times d$ real matrices, such that, for all smooth enough $\mathbb{R}^d$-valued function $f$,
\begin{align*}
\int_{\mathbb{R}^d} \langle x ; f(x)\rangle \rho(dx) =\int_{\mathbb{R}^d} \langle \tau_\rho(x); \nabla(f)(x) \rangle_{HS} \rho(dx),
\end{align*}
where $\langle A; B\rangle_{HS}=\operatorname{Tr}\left(A^t B\right)$, for $A,B\in \mathcal{M}_{d\times d}(\mathbb{R})$. Recall, also from the previous section, that a non-degenerate self-decomposable random vector $X$ without Gaussian component, with law $\mu_X$ and with finite first moment satisfies, for $f$ smooth enough, the following characterizing equation,
\begin{align*}
\int_{\mathbb{R}^d}\langle x-\bbe X; \nabla(f)(x) \rangle \mu_X(dx)= \int_{\mathbb{R}^d} \left(\int_{\mathbb{R}^d} \langle \nabla(f)(x+u)-\nabla(f)(x);u \rangle \nu(du)\right) \mu_X(dx).
\end{align*}
Then, quite naturally, in the infinitely divisible framework, let us introduce the following definitions of Stein's kernels and of the Stein's discrepancy.

\begin{defi}\label{defSteinkernel2}
Let $X$ be a centered non-degenerate infinitely divisible random vector without Gaussian component, with law $\mu_X$, with L\'evy measure $\nu$ and with $\bbe \|X\|^2<\infty$. Let $Y$ be a centered random vector with law $\mu_Y$ and with $\bbe \|Y\|^2<\infty$.
Then, a Stein kernel of $Y$ with respect to $X$ is a measurable function $\tau_Y$ from $\mathbb{R}^d$ to $\mathbb{R}^d$ such that,
\begin{align*}
\int_{\mathbb{R}^d}\langle y; f(y) \rangle \mu_Y(dy)=\int_{\mathbb{R}^d} \left(\int_{\mathbb{R}^d} \langle f(y+u)-f(y);\tau_Y(y+u)-\tau_Y(y) \rangle \nu(du)\right) \mu_Y(dy),
\end{align*}
for all $\mathbb{R}^d$-valued test function $f$ for which both sides of the previous equality are well defined. Moreover, the Stein's discrepancy of $\mu_Y$ with respect to $\mu_X$ is given by
\begin{align*}
S\left(\mu_Y || \mu_X\right)= \inf\left( \int_{\mathbb{R}^d}\int_{\mathbb{R}^d} \|\tau_Y(y+u)-\tau_Y(y)-u\|^2\nu(du)\mu_Y(dy)\right)^{1/2},
\end{align*}
where the infimum is taken over all Stein kernels of $Y$ with respect to $X$, and is equal to $+\infty$ if no such Stein kernel exists.
\end{defi}
\noindent
The next result ensures that the Stein's discrepancy provides a good control on classical metrics between probability measures on $\mathbb{R}^d$.

\begin{thm}\label{thm:SteinKernel}
Let $X$ be a centered non-degenerate self-decomposable random vector without Gaussian component, with law $\mu_X$, L\'evy measure $\nu$, such that $\bbe \|X\|^2<+\infty$ and let also the functions $k_x$ given by \eqref{rep:sd} satisfy \eqref{cond:kx}. Let $Y$ be a centered non-degenerate random vector with law $\mu_Y$, such that $\bbe \|Y\|^2<+\infty$, and for which a Stein's kernel with respect to $X$ exists. Then, 
\begin{align*}
d_{W_2}(\mu_X,\mu_Y) \leq \frac{d}{2} \left(\int_{\mathbb{R}^d} \|u\|^2 \nu(du)\right)^{1/2}S\left(\mu_Y || \mu_X\right).
\end{align*}
\end{thm}

\begin{proof}
Let $h\in \mathcal{H}_2\cap C_c^\infty(\mathbb{R}^d)$. By Proposition \ref{prop:SteinEq}, $f_h$ is a solution to,
\begin{align*}
-\langle x;\nabla (f_h)(x) \rangle+\int_{\mathbb{R}^d}\langle \nabla(f_h)(x+u)-\nabla(f_h)(x) ;u\rangle \nu(du)=h(x)-\bbe h(X), \quad\quad x\in \mathbb{R}^d,
\end{align*}
and thus,
\begin{align*}
\bbe \left(-\langle Y;\nabla (f_h)(Y) \rangle+\int_{\mathbb{R}^d}\langle \nabla(f_h)(Y+u)-\nabla(f_h)(Y) ;u\rangle \nu(du)\right)=\bbe h(Y)-\bbe h(X).
\end{align*}
Now, since $Y$ admits a Stein kernel with respect to $X$,
\begin{align*}
\bbe h(Y)-\bbe h(X)=\bbe \left(\int_{\mathbb{R}^d}\langle \nabla(f_h)(Y+u)-\nabla(f_h)(Y) ; u -\tau_{Y}(Y+u)+\tau_Y(Y) \rangle \nu(du)\right).
\end{align*}
Taking the absolute values and applying the Cauchy-Schwarz inequality,
\begin{align*}
\left| \bbe h(Y)-\bbe h(X)\right| \leq \bbe \int_{\mathbb{R}^d} \|\nabla(f_h)(Y+u)-\nabla(f_h)(Y)\| \|\tau_{Y}(Y+u)-\tau_Y(Y)-u\| \nu(du).                                     
\end{align*}
Now, by the very definition of $M_2(f_h)$ and by the Cauchy-Schwarz inequality (applied twice), the following bound holds true
\begin{align*}
\left| \bbe h(Y)-\bbe h(X)\right| \leq M_2(f_h)  \left(\int_{\mathbb{R}^d}\|u\|^2\nu(du)\right)^{1/2} \left(\bbe \int_{\mathbb{R}^d} \|\tau_{Y}(Y+u)-\tau_Y(Y)-u\|^2\nu(du)\right)^{1/2}.
\end{align*}
To conclude use the definition of the Stein's discrepancy and Proposition \ref{EstM}.
\end{proof}
\noindent
In the sequel, we wish to discuss sufficient conditions for the existence of Stein kernels as defined above. For this purpose,  let us recall some definition and results from \cite{Chen,ChLo87} regarding Poincar\'e inequalities in an infinitely divisible setting. First, if $X$ is a non-degenerate infinitely divisible random vector in $\mathbb{R}^d$ without Gaussian component, with law $\mu_X$ and with L\'evy measure $\nu$ and if $f: \mathbb{R}^d \longrightarrow \mathbb{R}$ is such that $\bbe f(X)^2+\bbe \int_{\mathbb{R}^d} |f(X+u)-f(X)|^2 \nu(du)<+\infty$, then \cite[Theorem $4.1$]{Chen} gives
\begin{align}\label{Poinc:ID}
\operatorname{Var}(f(X)) \leq \bbe \int_{\mathbb{R}^d} |f(X+u)-f(X)|^2 \nu(du).
\end{align}
Further, if $Y$ is a centered non-degenerate random vector in $\mathbb{R}^d$ such that $\bbe \|Y\|^2<+\infty$, if $\nu$ is a L\'evy measure in $\mathbb{R}^d$ such that $\int_{\mathbb{R}^d}\|u\|^2\nu(du)<+\infty$ and $\mathcal{H}_Y$ is the space of real valued functions $f$ on $\mathbb{R}^d$ such that $\bbe f(Y)^2<+\infty$ and $0<\bbe \int_{\mathbb{R}^d} |f(Y+u)-f(Y)|^2 \nu(du)<+\infty$, then the Poincar\'e constant $U(Y,\nu)$ defined as
\begin{align}\label{eq:UCh}
U(Y,\nu)=\sup_{f\in \mathcal{H}_Y} \dfrac{\operatorname{Var}(f(Y))}{\bbe \int_{\mathbb{R}^d} |f(Y+u)-f(Y)|^2 \nu(du)}
\end{align}
characterizes the proximity in law of $Y$ to a centered infinitely divisible random vector with finite second moment and L\'evy measure $\nu$. Indeed, \cite[Theorem $2.1$]{ChLo87} gives the following: $U(Y,\nu)\geq 1$ and $U(Y,\nu)=1$ if and only if the characteristic function of $Y$ is given by
\begin{align*}
\varphi_Y(\xi)=\exp\left(\int_{\mathbb{R}^d} \left(e^{i \langle\xi;u\rangle}-1-i \langle\xi;u\rangle\right)\nu(du)\right),\quad \xi\in\mathbb{R}^d,
\end{align*}
i.e., $Y\sim ID(b,0,\nu)$ with $b=-\int_{\|u\|\geq 1}u\nu(du)$.

In the Gaussian case, the existence of a Stein kernel for multivariate distributions has been investigated with the help of variational methods. Indeed, in \cite{CFP18}, under a spectral gap assumption, the existence of a Gaussian Stein kernel has been ensured thanks to the classical Lax-Milgram Theorem. Then, in view of \eqref{eq:UCh} and the associated characterization, it is natural to introduce the following variational setting: let $Y$ be a centered non-degenerate random vector with finite second moment and with law $\mu_Y$ and let $\nu$ be a L\'evy measure on $\mathbb{R}^d$ such that $\int_{\|u\|\geq 1}\|u\|^2 \nu(du)<+\infty$. Moreover, assume that $\nu\ast \mu_Y<< \mu_Y$, with $\nu \ast \mu_Y$ denoting the convolution of the two positive measures $\nu$ and $\mu_Y$. Now, let $H_{\nu}(\mu_Y)$ be the vector space of Borel measurable $\mathbb{R}^d$-valued functions on $\mathbb{R}^d$ such that $\int_{\mathbb{R}^d} \|f(y)\|^2 \mu_Y(dy)<+\infty$ and $\int_{\mathbb{R}^d\times \mathbb{R}^d} \|f(y+u)-f(y)\|^2 \nu(du)\mu_Y(dy)<+\infty$ and let $H_{\nu,0}(\mu_Y)$ be the subspace of $H_{\nu}(\mu_Y)$ such that $\bbe f(Y)=0$. (Two functions $f$ and $g$ of $H_{\nu}(\mu_Y)$ are identified as soon as $f=g$ $\mu_Y$-almost everywhere.)
Then, let us assume that $Y$ satisfies a Poincar\'e inequality of the following type: there exists a positive and finite constant $U_Y$ such that, for all $f\in H_{\nu}(\mu_Y)$
\begin{align}\label{ineq:Poin}
\bbe \|f(Y)-\bbe f(Y)\|^2\leq U_Y\, \bbe \int_{ \mathbb{R}^d} \|f(Y+u)-f(Y)\|^2 \nu(du).
\end{align}
In particular, note that if $Y$ satisfies the Poincar\'e inequality \eqref{eq:UCh}, then, for all $f\in H_\nu(\mu_Y)$ such that $f_j\in \mathcal{H}_Y$, $1\leq j\leq d$,
\begin{align*}
\bbe |f_j(Y)-\bbe f_j(Y)|^2 \leq U(Y,\nu) \bbe \int_{\mathbb{R}^d} |f_j(Y+u)-f_j(Y)|^2 \nu(du),
\end{align*}
so that $Y$ satisfies a Poincar\'e inequality in the sense of the Inequality $\eqref{ineq:Poin}$ with $U_Y=U(Y,\nu)$.

Moreover, let $A$ be the bilinear functional defined, for all test functions $f$ and $g$, by
\begin{align}\label{def:A}
A(f,g)=\bbe \int_{\mathbb{R}^d} \langle f(Y+u)-f(Y); g(Y+u)-g(Y)\rangle \nu(du),
\end{align}
and let $L$ be the linear functional defined, for all test functions $f$, by
\begin{align}\label{def:L}
L(f)=\bbe \langle Y; f(Y) \rangle.
\end{align}
Before solving the variational problem associated with $A$, $L$ and $H_{\nu}(\mu_Y)$, we need the following technical lemma. 

\begin{lem}\label{lem:Hilbert}
The vector space $H_{\nu}(\mu_Y)$ endowed with the bilinear functional 
\begin{align}
\langle f;g \rangle_{H_{\nu}(\mu_Y)}=\bbe \langle f(Y); g(Y) \rangle + A(f,g)
\end{align}
is a Hilbert space. Moreover, $A$, defined by \eqref{def:A}, is continuous on $H_{\nu}(\mu_Y)\times H_{\nu}(\mu_Y)$, coercive on $H_{\nu,0}(\mu_Y)$ while, $L$, defined by \eqref{def:L}, is continuous on $H_{\nu}(\mu_Y)$.
\end{lem}

\begin{proof}
First, it is clear that the bilinear symmetric functional $\langle \cdot ; \cdot\rangle_{H_{\nu}(\mu_Y)}$ is an inner product on $H_{\nu}(\mu_Y)$. Then, let $\|\cdot\|_{H_{\nu}(\mu_Y)}$ be the induced norm defined via $\|f\|^2_{H_{\nu}(\mu_Y)}=\bbe \|f(Y)\|^2+A(f,f)$, for all $f\in H_{\nu}(\mu_Y)$. Let us prove that $H_{\nu}(\mu_Y)$ endowed with this norm is complete. Let $(f_n)_{n\geq 1}$ be a Cauchy sequence in $H_{\nu}(\mu_Y)$. Therefore $(f_n)_{n\geq 1}$ is a Cauchy sequence in $L^2(\mu_Y)$, and there exists $f\in L^2(\mu_Y)$ such that $f_n\rightarrow f$, as $n\rightarrow +\infty$ in $L^2(\mu_Y)$. Now, pick a subsequence $(f_{n_k})_{k\geq 1}$ such that $f_{n_k}\rightarrow f$, $\mu_Y$-almost everywhere, as $k\rightarrow +\infty$. Fatou's lemma together with the assumption that $\nu \ast \mu_Y<<\mu_Y$ and the fact that $(f_n)_{n\geq 1}$ is a Cauchy sequence in $H_{\nu}(\mu_Y)$ (thus is bounded), imply that
\begin{align}\label{Fatou1}
A(f,f) \leq \underset{k\rightarrow+\infty}{\liminf} A(f_{n_k},f_{n_k}) \leq \underset{n\geq 1}{\sup} \|f_n\|^2_{H_{\nu}(\mu_Y)}<+\infty.
\end{align}
Hence, $f$ belongs to $H_{\nu}(\mu_Y)$. Another application of Fatou's lemma together with the fact that $(f_n)_{n\geq 1}$ is Cauchy in $H_\nu(\mu_Y)$ shows that $f_{n}\rightarrow f$ in $H_\nu(\mu_Y)$. Now, by the Cauchy-Schwarz inequality, for all $f,g\in H_\nu(\mu_Y)$,
\begin{align*}
|A(f,g)| &\leq \left(\bbe \int_{\mathbb{R}^d} \|f(Y+u)-f(Y)\|^2 \nu(du) \right)^{1/2} \left(\bbe \int_{\mathbb{R}^d} \|g(Y+u)-g(Y)\|^2 \nu(du) \right)^{1/2} \\
&\leq \|f\|_{H_\nu(\mu_Y)}  \|g\|_{H_\nu(\mu_Y)}.
\end{align*}
Moreover, since $Y$ satisfies the Poincar\'e inequality \eqref{ineq:Poin}, for all $f\in H_{\nu,0}(\mu_Y)$
\begin{align*}
A(f,f)&=\bbe \int_{\mathbb{R}^d} \|f(Y+u)-f(Y)\|^2 \nu(du),\\
&\geq \frac{1}{2}\bbe \int_{\mathbb{R}^d} \|f(Y+u)-f(Y)\|^2 \nu(du)+\frac{1}{2 U_Y}\bbe \|f(Y)\|^2,\\
&\geq C_{Y} \|f\|^2_{H_\nu(\mu_Y)}
\end{align*}
for $2C_Y= \min \left(1, 1/( U_Y)\right)>0$. Finally, the continuity property of the linear functional $L$ on $H_\nu(\mu_Y)$ follows from the Cauchy-Schwarz inequality, from $\bbe \|Y\|^2<+\infty$, and from the continuous embedding $H_\nu(\mu_Y)\hookrightarrow L^2(\mu_Y)$.
\end{proof}
\noindent
Note that since $H_{\nu,0}(\mu_Y)$ is a closed subspace of $H_{\nu}(\mu_Y)$, it is as well a Hilbert space with the inner product $\langle . ; . \rangle_{H_{\nu}(\mu_Y)}$. 

Based on Lemma \ref{lem:Hilbert}, a direct application of the Lax-Milgram Theorem ensures the existence of a Stein kernel in the sense of Definition \ref{defSteinkernel2} for probability measures $\mu_Y$ which satisfy the Poincar\'e inequality \eqref{ineq:Poin}. This is the content of the next theorem.

\begin{thm}\label{th:existence}
Let $Y$ be a centered non-degenerate random vector with finite second moment and with law $\mu_Y$ and let $\nu$ be a L\'evy measure on $\mathbb{R}^d$ such that $\int_{\|u\|\geq 1}\|u\|^2 \nu(du)<+\infty$. Assume that $Y$ satisfies the Poincar\'e inequality \eqref{ineq:Poin} for some $0<U_Y<+\infty$ and that $\nu \ast \mu_Y<< \mu_Y$. Then, there exists a unique $\tau_Y \in H_{\nu,0}(\mu_Y)$, such that, for all $f\in H_{\nu,0}(\mu_Y)$
\begin{align}\label{eq:VarProb}
A(f,\tau_Y)=L(f).
\end{align}
Moreover,
\begin{align}\label{est:Stein}
\bbe \int_{\mathbb{R}^d}\|\tau_Y(Y+u)-\tau_Y(Y)\|^2 \nu(du) \leq U_Y\bbe \|Y\|^2.
\end{align}
\end{thm}

\begin{proof}
The first part of the theorem is a direct application of the Lax-Milgram Theorem with $A$, $L$ and $H_{\nu,0}(\mu_Y)$. To obtain the inequality \eqref{est:Stein}, note that thanks to \eqref{eq:VarProb} with $f=\tau_Y$, 
\begin{align}
A(\tau_Y,\tau_Y)=\bbe \int_{\mathbb{R}^d} \|\tau_Y(Y+u)-\tau_Y(Y)\|^2 \nu(du)\leq \sqrt{\bbe \|Y\|^2} \sqrt{\bbe \|\tau_Y(Y)\|^2}.
\end{align}
Finally, the Poincar\'e inequality \eqref{ineq:Poin} combined with the previous inequality implies
\begin{align*}
\bbe \int_{\mathbb{R}^d} \|\tau_Y(Y+u)-\tau_Y(Y)\|^2 \nu(du)\leq (U_Y)^{1/2} \left(\bbe \|Y\|^2\right)^{1/2} \left(\bbe \int_{\mathbb{R}^d} \|\tau_Y(Y+u)-\tau_Y(Y)\|^2 \nu(du)\right)^{1/2}
\end{align*}
which concludes the proof.
\end{proof}
\noindent
The next theorem is the main result of this section.

\begin{thm}\label{th:PoinQuan}
Let $X$ be a centered non-degenerate self-decomposable random vector without Gaussian component, with law $\mu_X$, with L\'evy measure $\nu$, such that $\bbe \|X\|^2<+\infty$ and let also the functions $k_x$ given by \eqref{rep:sd} satisfy \eqref{cond:kx}. Let $Y$ be a centered non-degenerate random vector with law $\mu_Y$, with $\bbe \|Y\|^2<+\infty$ and such that $Y$ satisfies the Poincar\'e inequality \eqref{ineq:Poin} with $1\leq U_Y<+\infty$ and that $\nu \ast \mu_Y<< \mu_Y$. Then, 
\begin{align}\label{ineq:PoinQuan2}
d_{W_2}(\mu_X,\mu_Y) \leq \frac{d}{2} \left(\int_{\mathbb{R}^d} \|u\|^2 \nu(du)\right)^{1/2} \left(U_Y\bbe \|Y\|^2+\int_{\mathbb{R}^d} \|u\|^2\nu(du)-2\bbe \|Y\|^2\right)^{1/2}.
\end{align}
Moreover, if $\bbe \|Y\|^2=\int_{\mathbb{R}^d} \|u\|^2\nu(du)$, then
\begin{align}\label{ineq:PoinQuan1}
d_{W_2}(\mu_X,\mu_Y) \leq \frac{d}{2} \left(\int_{\mathbb{R}^d} \|u\|^2 \nu(du)\right)\sqrt{U_Y-1}.
\end{align}
\end{thm}

\begin{proof}
Let us start with the proof of \eqref{ineq:PoinQuan1}. First, note that since $M_1(f_h)<+\infty$ and $M_2(f_h)<+\infty$, for $h\in \mathcal{H}_2\cap C_c^\infty(\mathbb{R}^d)$, $\nabla(f_h)$ belongs to $H_{\nu}(\mu_Y)$ with $f_h$ given by Proposition \ref{prop:SteinEq}. Thus, using $\bbe Y=0$, by Theorem \ref{thm:SteinKernel}, 
\begin{align}\label{ineq:1QPoin}
d_{W_2}(\mu_X,\mu_Y) \leq \frac{d}{2} \left(\int_{\mathbb{R}^d} \|u\|^2 \nu(du)\right)^{1/2}S\left(\mu_Y || \mu_X\right).
\end{align}
We continue by estimating $\bbe \int_{\mathbb{R}^d} \|\tau_Y(Y+u)-\tau_Y(Y)-u\|^2\nu(du)$. By the Pythagorean Theorem, Definition \ref{defSteinkernel2} and the fact that $\bbe \|Y\|^2=\int_{\mathbb{R}^d} \|u\|^2\nu(du)$
\begin{align*}
\bbe \int_{\mathbb{R}^d} \|\tau_Y(Y+u)-\tau_Y(Y)-u\|^2\nu(du)&= \bbe \int_{\mathbb{R}^d} \|\tau_Y(Y+u)-\tau_Y(Y)\|^2\nu(du)+\int_{\mathbb{R}^d} \|u\|^2\nu(du)\\
&\quad\quad-2\, \bbe \int_{\mathbb{R}^d} \langle u; \tau_Y(Y+u)-\tau_Y(Y) \rangle \nu(du),\\
&=\bbe \int_{\mathbb{R}^d} \|\tau_Y(Y+u)-\tau_Y(Y)\|^2\nu(du)+\int_{\mathbb{R}^d} \|u\|^2\nu(du)\\
&\quad\quad-2\, \bbe \|Y\|^2,\\
&=\bbe \int_{\mathbb{R}^d} \|\tau_Y(Y+u)-\tau_Y(Y)\|^2\nu(du)-\int_{\mathbb{R}^d} \|u\|^2\nu(du).
\end{align*}
Moreover, \eqref{est:Stein} implies that
\begin{align*}
\bbe \int_{\mathbb{R}^d} \|\tau_Y(Y+u)-\tau_Y(Y)-u\|^2\nu(du) \leq (U_Y-1) \int_{\mathbb{R}^d} \|u\|^2\nu(du),
\end{align*}
so that
\begin{align}\label{ineq:2QPoin}
S\left(\mu_Y || \mu_X\right) \leq \sqrt{U_Y-1} \left(\int_{\mathbb{R}^d} \|u\|^2\nu(du)\right)^{1/2}.
\end{align}
Combining \eqref{ineq:1QPoin} and \eqref{ineq:2QPoin} concludes the proof of the theorem. The proof of \eqref{ineq:PoinQuan2} follows in a completely similar manner.
\end{proof}

\begin{rem}\label{rem:PoinQuan}
(i) When $\bbe \|Y\|^2=\int_{\mathbb{R}^d} \|u\|^2\nu(du)$ and $\bbe Y=0$, note that $U_Y\geq 1$. Indeed, in \eqref{ineq:Poin}, take $f(y)=y$, for all $y\in \mathbb{R}^d$.\\
(ii) If Y is as in Theorem \ref{th:PoinQuan} with $\bbe \|Y\|^2=\int_{\mathbb{R}^d} \|u\|^2\nu(du)$, and if $U_Y=1$, then, clearly from Theorem \ref{th:PoinQuan}, $Y=_{d}X$ since $d_{W_2}(\mu_X,\mu_Y)=0$. Conversely, if $Y=_d X$, with $X$ as in Theorem \ref{th:PoinQuan}, then, for all $f=(f_1,...,f_d)$, \eqref{Poinc:ID} asserts that
\begin{align}
\bbe |f_j(Y)-\bbe f_j(Y)|^2 \leq \bbe \int_{\mathbb{R}^d} |f_j(Y+u)-f_j(Y)|^2 \nu(du),
\end{align}
for all $1\leq j\leq d$. Therefore, it follows that $U_Y=1$.\\
(iii) The following inequality on the Stein discrepancy is a direct byproduct of the proof of the previous theorem
\begin{align*}
S\left(\mu_Y || \mu_X\right) \leq \left(U_Y\bbe \|Y\|^2+\int_{\mathbb{R}^d} \|u\|^2\nu(du)-2\bbe \|Y\|^2\right)^{1/2}.
\end{align*}
(iv) All the above results should be compared with the analogous Gaussian ones obtained in \cite{CFP18} (see \cite[Theorem $2.4$ and Corollary $2.5$]{CFP18}).
\end{rem}
\noindent
As a straightforward corollary to Theorem \ref{th:PoinQuan}, the following convergence result holds true.

\begin{cor}\label{Sequence}
Let $X$ be a centered non-degenerate self-decomposable random vector without Gaussian component, with law $\mu_X$, L\'evy measure $\nu$, such that $\bbe \|X\|^2<+\infty$ and let also the functions $k_x$ given by \eqref{rep:sd} satisfy \eqref{cond:kx}. Let $(Y_n)_{n\geq 1}$ be a sequence of centered square-integrable non-degenerate random vectors with laws $(\mu_{n})_{n\geq 1}$, such that $\nu \ast \mu_n << \mu_n$, for all $n\geq 1$, and such that $Y_n$ satisfies the Poincar\'e inequality \eqref{ineq:Poin} with $1\leq U_n<+\infty$, for all $n\geq 1$. If $\bbe \|Y_n\|^2 \rightarrow \int_{\mathbb{R}^d}\|u\|^2\nu(du)$ and $U_n \rightarrow 1$, as $n$ tends to $+\infty$, then, $(Y_n)_{n\geq 1}$ converges in distribution towards $X$.
\end{cor}
\noindent
To end this section, we briefly discuss the condition $\nu\ast \mu_Y << \mu_Y$ appearing in Theorems \ref{th:existence} and \ref{th:PoinQuan}. For this purpose, let $\nu$ be the L\'evy measure of a non-degenerate infinitely divisible random vector, $X$, in $\mathbb{R}^d$ with law $\mu_X$. Now, let $\mathcal{P}(\nu)$ be the set of probability measures, $\mu$, on $\mathbb{R}^d$, such that $\nu\ast\mu<<\mu$. First of all, thanks to \cite[Lemma $4.1$]{Chen}, the set $\mathcal{P}(\nu)$ is not empty and contains the probability measure $\mu_X$. Moreover, it is clearly a convex set. Now, let us describe some further non-trivial examples of probability measures belonging to $\mathcal{P}(\nu)$. For this purpose, we say that two probability measures $\mu_1$ and $\mu_2$ on $\mathbb{R}^d$ are equivalent (denoted by $\mu_1\sim \mu_2$) if for any Borel set $B$ of $\mathbb{R}^d$, $\mu_1(B)=0$ if and only if $\mu_2(B)=0$.

\begin{prop}\label{Pnu}
Let $X$ be a non-degenerate infinitely divisible random vector in $\mathbb{R}^d$ with law $\mu_X$ and L\'evy measure $\nu$ and $\mathcal{P}(\nu)$ be the set of probability measures, $\mu$, in $\mathbb{R}^d$ such that $\nu\ast \mu<<\mu$. Let $Y$ be a non-degenerate random vector in $\mathbb{R}^d$ with law $\mu_Y$ such that $\mu_Y \sim\mu_X$. Then, $\mu_Y\in \mathcal{P}(\nu)$.
\end{prop}

\begin{proof}
Let $B$ be a Borel set of $\mathbb{R}^d$ such that $\mu_Y(B)=0$. Hence $\mu_X(B)=0$ since $\mu_Y\sim \mu_X$. But $\mu_X\in \mathcal{P}(\nu)$, thus $\nu\ast\mu_X(B)=0$. Finally, $\nu\ast \mu_Y<<\nu\ast \mu_X$, since $\mu_Y\sim \mu_X$, and therefore, $\nu\ast \mu_Y(B)=0$, which concludes the proof. 
\end{proof}
\noindent
As a further straightforward corollary, the following result holds true.

\begin{cor}\label{ID:Pnu}
Let $X$ be a non-degenerate infinitely divisible random vector in $\mathbb{R}^d$ without Gaussian component, with law $\mu_X$, L\'evy measure $\nu_X$ and parameter $b_X\in\mathbb{R}^d$ and let $\mathcal{P}(\nu_X)$ be the set of probability measures, $\mu$, on $\mathbb{R}^d$ such that $\nu_X\ast \mu<<\mu$. Let $Y$ be a non-degenerate infinitely divisible random vector in $\mathbb{R}^d$ without Gaussian component, with law $\mu_Y$, L\'evy measure $\nu_Y$ and parameter $b_Y\in \mathbb{R}^d$. Assume that $\nu_X \sim \nu_Y$ and that
\begin{align*}
\int_{\mathbb{R}^d} \left(e^{\Phi(u)/2}-1\right)^2 \nu_X(du)<+\infty,\quad b_Y-b_X-\int_{\|u\|\leq 1} u(\nu_Y-\nu_X)(du)=0,
\end{align*} 
with $\exp(\Phi(u))=d\nu_Y/d\nu_X $, for all $u\in \mathbb{R}^d$. Then, $\mu_Y\in \mathcal{P}(\nu_X)$.
\end{cor}

\begin{proof}
This is a direct application of Proposition \ref{Pnu} together with \cite[Theorem $33.1$]{S}.
\end{proof}



\appendix
\section{Appendix}
\label{sec:appendix}
\noindent
The aim of this section is to provide technical results (often multivariate versions of univariate ones proved in \cite{AH18}) which are used throughout the previous sections.

\begin{lem}\label{lem:MomBounds}
Let $X$ be a non-degenerate self-decomposable random vector in $\mathbb{R}^d$, without Gaussian component, with law $\mu_X$, characteristic function $\varphi$ and such that $\bbe \|X\|<\infty$. Assume further that, for any $0<a<b<+\infty$ the functions $k_x$ given by \eqref{rep:sd} satisfy the following condition
\begin{align}\label{MB:condk}
\sup_{x\in S^{d-1}}\sup_{r\in (a,b)}k_x(r)<+\infty.
\end{align}
Let $X_t$, $t\geq 0$, be the random vectors each with characteristic functions, $\varphi_t$, given, for all $\xi\in \mathbb{R}^d$ by
\begin{align}
\varphi_t(\xi)=\dfrac{\varphi(\xi)}{\varphi(e^{-t}\xi)}.
\end{align}
Then,
\newline
(i)
\begin{align}\label{eq:finiteSup}
\underset{t> 0}{\sup}\,\bbe \|X_t\|<+\infty,
\end{align}
and,\\
(ii) for all $\xi\in \mathbb{R}^d$ and all $t\in(0,1)$,
\begin{align}\label{ineq:boundcharac}
\frac{1}{t}\left|\varphi_t(\xi)-1\right|\leq C (\|\xi\| \|\bbe X\|+\|\xi\|+\|\xi\|^2),
\end{align}
for some $C>0$ independent of $\xi$ and $t$.
\end{lem}

\begin{proof}
Let us start with the proof of $(i)$. First note that, for all $t> 0$
\begin{align*}
X_t=_d (1-e^{-t})\bbe X+ Y_t+ Z_t,
\end{align*}
where $Y_t$ and $Z_t$ are independent, with, for all $\xi\in \mathbb{R}^d$
\begin{align*}
&\bbe e^{i \langle \xi; Y_t \rangle}=\exp\left(\int_{u\in D}\left(e^{i \langle \xi; u\rangle}-1-i\langle \xi;u \rangle\right)\nu_t(du)\right)\\
&\bbe e^{i \langle \xi; Z_t \rangle}=\exp\left(\int_{u\in D^c}\left(e^{i \langle \xi; u\rangle}-1-i\langle \xi;u \rangle\right)\nu_t(du)\right),
\end{align*}
with $\nu_t$ the L\'evy measure of $X_t$. Then, for all $t> 0$,
\begin{align*}
\bbe \|X_t\| \leq (1-e^{-t})\bbe \|X\|+ \bbe \|Y_t\|+\bbe \|Z_t\|.
\end{align*}
Using \cite[Lemma 1.1]{MR01}, 
\begin{align*}
\bbe \|X_t\| \leq (1-e^{-t})\bbe \|X\|+\left(\int_{\|u\|\leq 1}\|u\|^2\nu_t(du)\right)^{1/2}+2\int_{\|u\|\geq 1}\|u\|\nu_t(du).
\end{align*}
Now, thanks to the representation \eqref{rep:charac2}, 
\begin{align*}
\int_{\|u\|\leq 1}\|u\|^2\nu_t(du)\leq \int_{\|u\|\leq 1}\|u\|^2\nu(du),\quad\quad\int_{\|u\|\geq 1}\|u\|\nu_t(du)\leq \int_{\|u\|\geq 1}\|u\|\nu(du).
\end{align*}
Thus, 
\begin{align*}
\sup_{t\geq 0}\bbe \|X_t\|\leq \bbe \|X\|+\left(\int_{\|u\|\leq 1}\|u\|^2\nu(du)\right)^{\frac{1}{2}}+2\int_{\|u\|\geq 1}\|u\|\nu(du)<+\infty.
\end{align*}
To prove (ii), first note that, for all $\xi\in \mathbb{R}^d$ with $\|\xi\|\ne 0$ and all $t>0$
\begin{align*}
\bbe e^{i \langle \xi;X_t \rangle}-1=\int_0^{\|\xi\|} \left\langle \nabla(\varphi_t)\left(s\frac{\xi}{\|\xi\|}\right);\dfrac{\xi}{\|\xi\|}\right\rangle ds,
\end{align*}
and thus
\begin{align*}
\left|\bbe e^{i \langle \xi;X_t \rangle}-1\right|\leq \|\xi\| \max_{s\in [0,\|\xi\|]}\left\|\nabla(\varphi_t)\left(s\frac{\xi}{\|\xi\|}\right)\right\|.
\end{align*}
Noting that, for all $\xi\in \mathbb{R}^d$ and all $1\leq j\leq d$,
\begin{align}
\partial_j(\varphi_t)(\xi)=\left(i\bbe X_j(1-e^{-t})+i\int_{\mathbb{R}^d}u_j \left(e^{i\langle u;\xi \rangle}-1\right)\nu_t(du)\right)\varphi_t(\xi),
\end{align}
it follows that
\begin{align*}
\left|\bbe e^{i \langle \xi;X_t \rangle}-1\right|\leq \|\xi\| (1-e^{-t})\sum_{j=1}^d\bbe |X_j|+ \sqrt{d}\|\xi\|^2 \int_{\|u\|\leq 1}\|u\|^2\nu_t(du)+2\|\xi\|\sqrt{d}\int_{\|u\|\geq 1}\|u\|\nu_t(du).
\end{align*}
Then, the polar decomposition of $\nu_t$, allows to bound the two terms $\int_{\|u\|\leq 1}\|u\|^2\nu_t(du)$ and $\int_{\|u\|\geq 1}\|u\|\nu_t(du)$. Let us first deal with the term $\int_{\|u\|\geq 1}\|u\|\nu_t(du)$. By \eqref{rep:charac2},
\begin{align*}
\int_{\|u\|\geq 1}\|u\|\nu_t(du)&=\int_{S^{d-1}\times (1,+\infty)} r\frac{k_x(r)-k_x(e^tr)}{r}dr\lambda(dx)\\
&=\int_{S^{d-1}} \left(\int_1^{e^t}k_x(r)dr+(1-e^{-t})\int_{e^t}^{+\infty}k_x(r)dr\right)\lambda(dx)\\
&\leq (e^t-1) \sup_{x\in S^{d-1}}|k_x(1^+)|+(1-e^{-t})\int_{\|u\|\geq 1}\|u\|\nu(du)
\end{align*}
which is finite in view of \eqref{MB:condk}. For the term $\int_{\|u\|\leq 1}\|u\|^2\nu_t(du)$,
\begin{align*}
\int_{\|u\|\leq 1}\|u\|^2\nu_t(du)&=\int_{S^{d-1}\times (0,1)} r(k_x(r)-k_x(e^tr))dr\lambda(dx)\\
&=\int_{S^{d-1}} \left(\int_0^1r(k_x(r)-k_x(e^tr))dr\right) \lambda(dx)\\
&=\int_{S^{d-1}} \left(-e^{-2t}\int_1^{e^t}rk_x(r)dr+(1-e^{-2t})\int_0^{1}rk_x(r)dr\right)\lambda(dx)\\
&\leq (1-e^{-2t})\int_{S^{d-1}\times (0,1)}rk_x(r)dr\lambda(dx).
\end{align*}
This concludes the proof of the lemma.
\end{proof}
\noindent
\begin{lem}\label{lem:repsmooth0}
Let $X,Y$ be two random vectors in $\mathbb{R}^d$ with respective laws $\mu_X$ and $\mu_Y$. Let $r\geq 1$. Then,
\begin{align}
d_{W_r}(\mu_X,\mu_Y)=\underset{h\in \mathcal{H}_r\cap \mathcal{C}^{\infty}_c(\mathbb{R}^d)}{\sup} \left|\bbe h(X)-\bbe h(Y)\right|.
\end{align}
\end{lem}
\begin{proof}
Let $r\geq 1$. First, it is clear that 
\begin{align*}
d_{W_r}(\mu_X,\mu_Y)\geq \underset{h\in \mathcal{H}_r\cap \mathcal{C}^{\infty}_c(\mathbb{R}^d)}{\sup} \left|\bbe h(X)-\bbe h(Y)\right|.
\end{align*}
Now, let $h\in \mathcal{H}_r$ and let $(h_\varepsilon)_{\varepsilon>0}$ be the regularization of $h$ with the Gaussian kernel, namely, for all $x\in \mathbb{R}^d$ and all $\varepsilon>0$
\begin{align*}
h_\varepsilon(x):=\int_{\mathbb{R}^d} h(x-y)\exp\left(-\frac{\|y\|^2}{2\varepsilon^2}\right)\dfrac{dy}{(2\pi)^{\frac{d}{2}}\varepsilon^d}
\end{align*}
Note that $h_\varepsilon \in \mathcal{C}^{\infty}(\mathbb{R}^d)$, for all $\varepsilon>0$. Moreover, 
\begin{align*}
\|h-h_\varepsilon\|_\infty\leq d \varepsilon,\quad\quad M_{\ell}(h_\varepsilon)\leq 1,\quad  0\leq \ell\leq r.
\end{align*}
Next, let $\Psi$ be a compactly supported infinitely differentiable function with values in $[0,1]$ such that $\operatorname{supp}(\Psi)\subseteq D(0,2)$ and such that $\Psi(x)=1$, for all $x\in D$. Then, for any $R\geq 1$ and any $\varepsilon>0$, set, for all $x\in \mathbb{R}^d$
\begin{align}
h_{\varepsilon, R}(x):=\Psi\left(\frac{x}{R}\right)h_\varepsilon(x).
\end{align}
Then, for $X$ and $Y$ two random vectors on $\mathbb{R}^d$ with respective laws $\mu_X$ and $\mu_Y$,\begin{align*}
\left| \bbe h(X)- \bbe h(Y)\right|&\leq \left| \bbe h_{\varepsilon, R}(X) - \bbe h_{\varepsilon, R}(Y)\right|+2d\varepsilon + \int_{\mathbb{R}^d} \left(1-\Psi\left(\frac{x}{R}\right)\right)d\mu_X(x)\\
&\quad \quad+\int_{\mathbb{R}^d} \left(1-\Psi\left(\frac{x}{R}\right)\right)d\mu_Y(x),\\
&\leq \left| \bbe h_{\varepsilon, R}(X) - \bbe h_{\varepsilon, R}(Y)\right|+2d\varepsilon +\mathbb{P}\left(\|X\|\geq R\right)+\mathbb{P}\left(\|Y\|\geq R\right).
\end{align*}
Now, for $R\geq 1$ such that $\max \left(\mathbb{P}\left(\|X\|\geq R\right),\mathbb{P}\left(\|Y\|\geq R\right)\right)\leq \varepsilon$, 
\begin{align*}
\left| \bbe h(X)- \bbe h(Y)\right|\leq \left| \bbe h_{\varepsilon, R}(X) - \bbe h_{\varepsilon, R}(Y)\right|+(2d+2)\varepsilon.
\end{align*}
To continue, one needs to estimate the quantities $M_\ell (h_{\varepsilon, R})$, for all $0\leq \ell \leq r$. First, since $h\in \mathcal{H}_r$
\begin{align*}
M_0(h_{\varepsilon, R}):=\sup_{x\in \mathbb{R}^d}\left|h_{\varepsilon, R}(x)\right| \leq 1.
\end{align*}
Now, for $v\in \mathbb{R}^d$ such that $\|v\|=1$ and $x\in \mathbb{R}^d$
\begin{align*}
\mathbf{D}(h_{\varepsilon, R})(v)(x)&=\sum_{i=1}^d v_i\partial_i \left(h_{\varepsilon, R}\right)(x)\\
&=\frac{1}{R}\sum_{i=1}^d v_i h_\varepsilon(x) \partial_i(\Psi)\left(\frac{x}{R}\right)+\Psi\left(\frac{x}{R}\right)\sum_{i=1}^d v_i \partial_i \left(h_\varepsilon\right)(x)\\
&=\frac{h_\varepsilon(x)}{R}\langle \nabla \left(\Psi\right)\left(\frac{x}{R}\right);v\rangle+\Psi\left(\frac{x}{R}\right) \langle \nabla (h_\varepsilon)(x) ;v\rangle.
\end{align*}
Then, for all $R\geq 1$ and all $\varepsilon>0$
\begin{align}
M_1(h_{\varepsilon, R})\leq \frac{1}{R} \sup_{x\in \mathbb{R}^d} \|\nabla \left(\Psi\right)\left(x\right)\|+1.
\end{align}
By a similar reasoning, it follows that
$M_\ell(h_{\varepsilon, R})\leq C_{\ell,\Psi}\left(\sum_{k=1}^\ell 1/R^k\right)+1$, for all $1\leq \ell \leq r$ and for some $C_{\ell,\Psi}>0$ only depending on $\ell$ and $\Psi$. Then, the function $\widetilde{h}_{\varepsilon, R}$ defined, for all $x\in \mathbb{R}^d$, by
\begin{align*}
\widetilde{h}_{\varepsilon, R}(x):= \dfrac{h_{\varepsilon, R}(x)}{\max_{1\leq \ell \leq r}(C_{\ell,\Psi}) \left(\sum_{k=1}^r 1/R^k\right)+1},
\end{align*}
belongs to $\mathcal{H}_r \cap \mathcal{C}_c^\infty(\mathbb{R}^d)$. Thus,
\begin{align*}
\left| \bbe h(X)- \bbe h(Y)\right|&\leq \left(\max_{1\leq \ell \leq r}(C_{\ell,\Psi}) \left(\sum_{k=1}^r \frac{1}{R^k}\right)+1\right)\left| \bbe \widetilde{h}_{\varepsilon, R}(X) - \bbe \widetilde{h}_{\varepsilon, R}(Y)\right|+(2d+2)\varepsilon\\
&\leq \left(\max_{1\leq \ell \leq r}(C_{\ell,\Psi}) \left(\sum_{k=1}^r \frac{1}{R^k}\right)+1\right) \underset{h\in \mathcal{H}_r\cap \mathcal{C}^{\infty}_c(\mathbb{R}^d)}{\sup} \left|\bbe h(X)-\bbe h(Y)\right|+ (2d+2)\varepsilon.
\end{align*}
Letting first $R$ tend to $+\infty$ and then $\varepsilon$ tend to $0^+$ concludes the proof of the lemma.
\end{proof}
\noindent
The objective of Theorem \ref{thm:smoothing} below is to prove that the $d_{W_1}$ distance between the law of $X$ and the law of $X_t$ decreases exponentially fast as $t$ tends to $+\infty$. For this purpose, for any $r\geq 1$ and any random vectors $X$ and $Y$, let
\begin{align}\label{eq:wassermod}
d_{\widetilde{W}_r}(X,Y)=\underset{h\in \widetilde{\mathcal{H}}_r}{\sup}\left|\bbe h(X)-\bbe h(Y)\right|,
\end{align}
where $\widetilde{\mathcal{H}}_r$ is the set of functions which are $r$-times continuously differentiable on $\mathbb{R}^d$ such that $\|D^\alpha(f)\|_{\infty}\leq 1$, for all $\alpha\in \mathbb{N}^d$ with $0\leq |\alpha|\leq r$. Since, for any $r\geq 1$, $\mathcal{H}_r\subset \widetilde{\mathcal{H}}_r$,
\begin{align}\label{ineq:WtildeW}
d_{W_r}(X,Y)\leq d_{\widetilde{W}_r}(X,Y).
\end{align}
The next lemma shows that smooth compactly supported function in $\widetilde{H}_r$, $r\geq 1$, are enough in \eqref{eq:wassermod}.


\begin{lem}\label{lem:repsmooth}
Let $X,Y$ be two random vectors in $\mathbb{R}^d$ with respective laws $\mu_X$ and $\mu_Y$. Let $r\geq 1$. Then,
\begin{align}
d_{\tilde{W}_r}(\mu_X,\mu_Y)=\underset{h\in \widetilde{\mathcal{H}}_r\cap \mathcal{C}^{\infty}_c(\mathbb{R}^d)}{\sup} \left|\bbe h(X)-\bbe h(Y)\right|.
\end{align}
\end{lem}

\begin{proof}
Let $r\geq 1$. By definition,
\begin{align}
d_{\tilde{W}_r}(\mu_X,\mu_Y)\geq \underset{h\in \widetilde{\mathcal{H}}_r\cap \mathcal{C}^{\infty}_c(\mathbb{R}^d)}{\sup} \left|\bbe h(X)-\bbe h(Y)\right|.
\end{align}
Let $h\in \widetilde{\mathcal{H}}_r$ and let $(h_\varepsilon)_{\varepsilon>0}$ be a regularization by convolution of $h$, such that $h_\varepsilon\in \mathcal{C}^\infty(\mathbb{R}^d)$ and 
\begin{align*}
\|h-h_\varepsilon\|_\infty\leq d\varepsilon,\quad\quad \|D^{\alpha}(h_{\varepsilon})\|_\infty\leq 1,\quad \alpha\in \mathbb{N}^d,\quad 0\leq |\alpha|\leq r.
\end{align*}
Let $\psi$ be a compactly supported, even, infinitely differentiable function on $\mathbb{R}$ with values in $[0,1]$ such that $\psi(x)=1$, for $x\in [-1,1]$. Then, for all $M\geq 1$, $\varepsilon>0$ and $x\in \mathbb{R}^d$ set $\Psi_M(x)=\prod_{i=1}^d \psi(x_i/M)$ and set also,
\begin{align*}
h_{M,\varepsilon}(x)=\Psi_M(x)h_\varepsilon(x).
\end{align*}
Clearly, by construction, $h_{M,\varepsilon}\in \mathcal{C}_c^\infty(\mathbb{R}^d)$. Then, for all $M\geq 1$ and $\varepsilon>0$,
\begin{align*}
|\bbe h(X)-\bbe h(Y)|\leq |\bbe h_{M,\varepsilon}(X)-\bbe h_{M,\varepsilon}(Y) |+2d\varepsilon+\int_{\mathbb{R}^d}\left|1-\Psi_M(x)\right|d\mu_X(x)+\int_{\mathbb{R}^d}\left|1-\Psi_M(y)\right|d\mu_Y(y).
\end{align*}
Choosing $M\geq 1$ large enough,
\begin{align*}
|\bbe h(X)-\bbe h(Y)|\leq |\bbe h_{M,\varepsilon}(X)-\bbe h_{M,\varepsilon}(Y) |+(2d+2)\varepsilon.
\end{align*}
Next, by the very definition of $h_{M,\varepsilon}$
\begin{align*}
\|h_{M,\varepsilon}\|_\infty\leq 1,
\end{align*}
and, moreover, by Leibniz formula, for all $\alpha \in \mathbb{N}^d$ with $1\leq |\alpha|\leq r$ and $x\in \mathbb{R}^d$
\begin{align*}
|D^\alpha(h_{M,\varepsilon})(x)|&\leq \sum_{\beta\leq \alpha} \binom{\alpha}{\beta} |D^{\beta}(\Psi_M)(x)||D^{\alpha-\beta}(h_\epsilon)(x)|\\
&\leq |D^{\alpha}(h_\epsilon)(x)|+\sum_{\beta\leq \alpha,\, \beta\ne 0} \binom{\alpha}{\beta} |D^{\beta}(\Psi_M)(x)||D^{\alpha-\beta}(h_\epsilon)(x)|\\
&\leq 1+\sum_{\beta\leq \alpha,\, \beta\ne 0} \binom{\alpha}{\beta}|D^{\beta}(\Psi_M)(x)|.
\end{align*}
Now, for all $\beta\leq \alpha$, $\beta\ne 0$ and $x\in\mathbb{R}^d$
\begin{align*}
|D^{\beta}(\Psi_M)(x)|\leq \frac{1}{M^{|\beta|}} \prod_{1\leq j\leq d} \underset{x\in \mathbb{R}}{\sup}|\psi^{(\beta_j)}(x)|.
\end{align*}
Thus,
\begin{align*}
|D^\alpha(h_{M,\varepsilon})(x)|&\leq1+C_{d,\alpha}\sum_{\beta\leq \alpha,\, \beta\ne 0} \frac{1}{M^{|\beta|}},
\end{align*}
for some $C_{d,\alpha}>0$ only depending on $d$, $\alpha$ and $\psi$. This implies that
\begin{align*}
|\bbe h(X)-\bbe h(Y)|\leq \left(1+C_{d,r}\sum_{1\leq |\alpha|\leq r}\sum_{\beta\leq \alpha,\, \beta\ne 0} \frac{1}{M^{|\beta|}}\right)  \underset{h\in \widetilde{\mathcal{H}}_r\cap C^{\infty}_c(\mathbb{R}^d)}{\sup} \left|\bbe h(X)-\bbe h(Y)\right|+(2d+2)\varepsilon
\end{align*}
for some $C_{d,r}>0$ only depending on $d$, $r$ and $\psi$. The conclusion follows by, first taking $M\rightarrow +\infty$, and then $\varepsilon\rightarrow 0^+$.
\end{proof}

\begin{thm}\label{thm:smoothing}
Let $X$ be a non-degenerate self-decomposable random vector in $\mathbb{R}^d$, without Gaussian component, with law $\mu_X$, characteristic function $\varphi$ and such that $\bbe \|X\|<\infty$. Assume further that, for any $0<a<b<+\infty$ the functions $k_x$ given by \eqref{rep:sd} satisfy the following condition
\begin{align}
\sup_{x\in S^{d-1}}\sup_{r\in (a,b)}k_x(r)<+\infty.
\end{align}
Let $X_t$, $t> 0$ be random vectors each with law $\mu_{X_t}$, with characteristic function $\varphi_t$, given, for all $\xi\in \mathbb{R}^d$ by
\begin{align}
\varphi_t(\xi)=\dfrac{\varphi(\xi)}{\varphi(e^{-t}\xi)}.
\end{align}
Then, for $t> 0$
\begin{align}
d_{W_{1}}(\mu_{X_t},\mu_X)\leq C_d e^{-\frac{t}{2^{d+1}(d+1)}},
\end{align}
for some $C_d>0$ independent of $t$.
\end{thm}

\begin{proof}
\textit{Step 1}: Let $r\geq 2$ and let $h\in\widetilde{\mathcal{H}}_{r-1}$. Let $(h_{\varepsilon})_{\varepsilon>0}$ be a regularization by convolution of $h$ such that
\begin{align*}
\|h-h_\varepsilon\|_{\infty} \leq d \varepsilon,\ \quad \|D^{\alpha}(h_\varepsilon)\|_\infty\leq 1,\,\quad 0\leq |\alpha|\leq r-1.
\end{align*}
For $\alpha\in \mathbb{N}^d$ such that $|\alpha|=r$, let us estimate $\|D^{\alpha}(h_\varepsilon)\|_{\infty}$.
By definition, for all $x\in \mathbb{R}^d$,
\begin{align*}
h_{\varepsilon}(x)=\int_{\mathbb{R}^d}h(y)\exp\left(-\frac{\|x-y\|^2}{2\varepsilon^2}\right)\frac{dy}{(2\pi)^{\frac{d}{2}}\varepsilon^d}
\end{align*}
Now, by Rodrigues formula, for all $j\in 1,..., d$,
\begin{align*}
\partial_{x_j}^{\alpha_j}\left(\exp\left(-\frac{x_j^2}{2}\right)\right)=(-1)^{\alpha_j}H_{\alpha_j}(x_j)\exp\left(-\frac{x_j^2}{2}\right).
\end{align*}
where $H_{\alpha_j}$ is the Hermite polynomial of degree $\alpha_j$. Thus, for all $\alpha\in \mathbb{N}^d$ and $x\in \mathbb{R}^d$,
\begin{align*}
D^{\alpha}\left(\exp\left(-\frac{\| x \|^2}{2}\right)\right)=(-1)^{\alpha}H_{\alpha}(x) \exp\left(-\frac{\| x \|^2}{2}\right),
\end{align*}
where $H_\alpha(x)=\prod_{j=1}^d H_{\alpha_j}(x_j)$. Then, for all $\alpha\in \mathbb{N}^d$ such that $|\alpha|=r$, for all $x\in \mathbb{R}^d$ and for some $\beta\in \mathbb{N}^d$ such that $|\beta|=r-1$ and $\alpha-\beta\geq 0$
\begin{align*}
&D^\alpha(h_{\varepsilon})(x)=\int_{\mathbb{R}^d} D^\beta(h)(y)D^{\alpha-\beta}\left(\exp\left(-\frac{\| x-y \|^2}{2\varepsilon^2}\right)\right) \frac{dy}{(2\pi)^{\frac{d}{2}}\varepsilon^d},\\
&D^\alpha(h_{\varepsilon})(x)=\frac{(-1)}{\varepsilon}\int_{\mathbb{R}^d}D^\beta(h)(y) H_{\alpha-\beta}\left(\frac{x-y}{\varepsilon}\right)\exp\left(-\frac{\| x-y \|^2}{2\varepsilon^2}\right)\frac{dy}{(2\pi)^{\frac{d}{2}}\varepsilon^d}.
\end{align*}
Thus,
\begin{align*}
\|D^\alpha(h_{\varepsilon})\|_\infty &\leq \frac{1}{\varepsilon} \int_{\mathbb{R}^d}|H_{\alpha-\beta}(y)|\exp\left(-\frac{\|y\|^2}{2}\right)\frac{dy}{(2\pi)^{\frac{d}{2}}}\\
&\leq C_\alpha \varepsilon^{-1}\leq C_r \varepsilon^{-1},
\end{align*}
for some $C_\alpha,C_r>0$ depending only on $\alpha$, on $d$ and on $r$. Let $Z$ and $Y$ be two random vectors with respective laws $\mu_Z$ and $\mu_Y$ such that $d_{\widetilde{W}_r}(Z,Y)<1$. Then,
\begin{align*}
|\bbe h(Z)-\bbe h(Y)| \leq 2 d\varepsilon+|\bbe h_\varepsilon(Z)-\bbe h_\varepsilon(Y)|.
\end{align*}
Choosing $\varepsilon \in (0, C_r)$,
\begin{align*}
|\bbe h(Z)-\bbe h(Y)|&\leq 2 d\varepsilon+\frac{C_r}{\varepsilon} d_{\widetilde{W}_r}(Z,Y)\\
&\leq \max (2 d,C_r)\left(\varepsilon+\varepsilon^{-1}d_{\widetilde{W}_r}(Z,Y)\right).
\end{align*}
Taking $\varepsilon \leq C_r/(1+C_r)\sqrt{d_{\widetilde{W}_r}(Z,Y)}$ yields,
\begin{align*}
d_{\widetilde{W}_{r-1}}(Z,Y)\leq \tilde{C}_{r}\sqrt{d_{\widetilde{W}_r}(Z,Y)}
\end{align*}
for some $\tilde{C}_{r}>0$ only depending on $r$ and on $d$. Now, let $Z$ and $Y$ be two random vectors such that $d_{\widetilde{W}_2}(Z,Y)<1$. Then, thanks to \eqref{ineq:wasser}, $d_{\widetilde{W}_m}(Z,Y)<1$, for all $2\leq m\leq r$. By induction, we get
\begin{align}\label{ineq:W1Wr}
d_{\widetilde{W}_1}(Z,Y)\leq \overline{C}_r \left(d_{\widetilde{W}_r}(Z,Y)\right)^{\frac{1}{2^{r-1}}},
\end{align}
for some $\overline{C}_r>0$ only depending on $r$ and on $d$.\\
\\
\textit{Step 2}: Let $g$ be an infinitely differentiable function with compact support contained in the Euclidean ball centered at the origin of radius $R+1$, for some $R>0$. Then by Fourier inversion and Fubini theorem, for all $t>0$,
\begin{align*}
|\bbe g(X)-\bbe g(X_t)|&\leq e^{-t}\bbe \|X\| \frac{1}{(2\pi)^d}\int_{\mathbb{R}^d} |\mathcal{F}(g)(\xi)|\|\xi\| d\xi\\
&\leq e^{-t}\bbe \|X\| \frac{1}{(2\pi)^d}\int_{\mathbb{R}^d} |\mathcal{F}(g)(\xi)|\dfrac{(1+\|\xi\|)^{d+2}}{(1+\|\xi\|)^{d+2}}\|\xi\| d\xi\\
&\leq e^{-t}\bbe \|X\| \underset{\xi \in \mathbb{R}^d}{\sup}\left(|\mathcal{F}(g)(\xi)| (1+\|\xi\|^{d+2}) \right) \left(\frac{1}{(2\pi)^d}\int_{\mathbb{R}^d}\dfrac{\|\xi\|d\xi}{(1+\|\xi\|)^{d+2}}\right).
\end{align*}
Moreover, for all $p\geq 2$
\begin{align*}
\underset{\xi \in \mathbb{R}^d}{\sup}\bigg(|\mathcal{F}(g)(\xi)| (1+\|\xi\|^{p}) \bigg)\leq C_d (R+1)^d \left(\|g\|_\infty+\underset{1\leq j\leq d}{\max} \|\partial^p_j(g)\|_\infty\right),
\end{align*}
for some $C_d>0$ depending on the dimension $d$ only. Thus, for all $t>0$
\begin{align}\label{ineq:compactsupport}
|\bbe g(X)-\bbe g(X_t)|&\leq \tilde{C}_d e^{-t}\bbe \|X\| (R+1)^d \left(\|g\|_\infty+\underset{1\leq j\leq d}{\max} \|\partial^{d+2}_j(g)\|_\infty\right).
\end{align}
\textit{Step 3}: Let $h\in\mathcal{C}_c^\infty(\mathbb{R}^d)\cap \widetilde{\mathcal{H}}_{d+2}$. Let $\Psi_R$ be a compactly supported infinitely differentiable function on $\mathbb{R}^d$ whose support is contained in the Euclidean ball centered at the origin of radius $R+1$ with values in [0,1] and such that $\Psi_R(x)=1$, for all $x$ such that $\|x\|\leq R$. Then, for all $t>0$
\begin{align*}
|\bbe h(X)-\bbe h(X_t)|\leq &|\bbe h(X)\psi_R(X) -\bbe h(X_t)\Psi_R(X_t)|+|\bbe h(X)(1-\Psi_R(X))|\\
&\quad\quad+|\bbe h(X_t)(1-\Psi_R(X_t))|.
\end{align*}
Now, note that
\begin{align*}
|\bbe h(X_t)(1-\Psi_R(X_t))|&\leq \int_{\mathbb{R}^d}(1-\Psi_R(x))d\mu_t(x)\\
&\leq \mathbb{P}\left(\|X_t\|\geq R\right)\\
&\leq \dfrac{\bbe \|X_t\|}{R}\\
&\leq \dfrac{1}{R}\sup_{t>0}\bbe \|X_t\|,
\end{align*}
using Lemma \ref{lem:MomBounds}. A similar bound holds true for $|\bbe h(X)(1-\Psi_R(X))|$. Moreover, from \eqref{ineq:compactsupport}, 
\begin{align*}\label{ineq:inter}
|\bbe h(X)-\bbe h(X_t)|\leq \frac{C_d}{R}+\tilde{C}_d e^{-t}\bbe \|X\| (R+1)^d \left(\|h\Psi_R\|_\infty+\underset{1\leq j\leq d}{\max}\|\partial^{d+2}_j(h\Psi_R)\|_\infty\right),
\end{align*}
for some constant $C_d$ depending on $d$. Now, 
\begin{align*}
\|h\Psi_R\|_\infty\leq 1,
\end{align*}
and, by taking for $\Psi_R$ an appropriate tensorization of one dimensional bump functions $\psi_R$, 
\begin{align*}
\underset{1\leq j\leq d}{\max}\|\partial^{d+2}_j(h\Psi_R)\|_\infty\leq D,
\end{align*}
for some $D>0$ independent of $R$ and $h$. Then,
\begin{align*}
|\bbe h(X)-\bbe h(X_t)|\leq C_d\left( \frac{1}{R}+ (R+1)^d e^{-t}\bbe \|X\|\right).
\end{align*}
Choosing $R=e^{t/(d+1)}$, for all $t>0$, it follows that
\begin{align*}
d_{\widetilde{W}_{d+2}}(X,X_t)\leq \tilde{C}_d e^{-\frac{t}{d+1}},
\end{align*}
for some $\tilde{C}_d>0$. Using \eqref{ineq:W1Wr} with $r=d+2$, 
\begin{align*}
d_{\widetilde{W}_1}(X,X_t)\leq \overline{C}_d \left(d_{\widetilde{W}_{d+2}}(X,X_t)\right)^{\frac{1}{2^{d+1}}}\leq C_d e^{-\frac{t}{2^{d+1}(d+1)}}.
\end{align*}
Using the inequality \eqref{ineq:WtildeW} concludes the proof of the theorem.
\end{proof}

\end{document}